\documentclass[11pt,reqno]{amsart}
\usepackage[utf8]{inputenc}
\usepackage{amssymb,a4wide,esint,enumitem,plain}
\usepackage{color}
\usepackage[colorlinks=true]{hyperref}
\hypersetup{urlcolor=cyan,citecolor=cyan}

\setlist[enumerate]{wide,label=\emph{(\roman*)}}

\usepackage{environ}

\newtheorem{theorem}{Theorem}[section]
\newtheorem{lemma}[theorem]{Lemma}
\newtheorem{proposition}[theorem]{Proposition}
\newtheorem{claim}[theorem]{Claim}
\theoremstyle{definition}

\numberwithin{equation}{section}
\def\psl#1#2{\left(#1,#2\right)}

\def\pshbbk#1#2{\left(#1,#2\right)_{\dot H^1_{{\boldsymbol{\ell}}_k}}}
\def\pshbbun#1#2{\left(#1,#2\right)_{\dot H^1_{{\boldsymbol{\ell}}_1}}}

\def\Nsol{\mathcal N_{\Omega^C}}
\def\bell{\boldsymbol
	\ell}
\begin{document}
	
	\title[Multi-solitons for critical wave equation]{On multi-solitons for the energy-critical   wave equation in  dimension 5}
	
	\author{XU YUAN}
	
	\address{CMLS, \'Ecole polytechnique, CNRS, Institut Polytechnique de Paris, F-91128 Palaiseau Cedex, France.}
	
	\email{xu.yuan@polytechnique.edu}
\begin{abstract}
We construct $K$-solitons of the focusing energy-critical nonlinear wave equation in five-dimensional space, i.e. solutions $u$ of the equation such that
\begin{equation*}
\|\nabla_{t,x}u(t)-\nabla_{t,x}\big(\sum_{k=1}^{K}W_{k}(t)\big)\|_{L^{2}}\to 0\quad \mathrm{as}\ t\to \infty,
\end{equation*}
where $K\geq 2$ and for any $k\in \{1,\dots,K\}$, $W_{k}$ is the Lorentz transform of the explicit standing soliton $W(x)=(1+|x|^{2}/15)^{-3/2}$, with any speed $\bell_k\in \mathbb{R}^{5}$, $|\bell_{k}|<1$ satisfying $\bell_{k'}\ne \bell_{k}$ for $k'\ne k$, and  an explicit smallness condition. The proof extends the refined method of construction of asymptotic multi-solitons from~\cite{MMwave1,MMwave2}.
\end{abstract}
\maketitle
\section{Introduction}
\subsection{Statement of the main result.} 	In this paper, we are interested in the construction of multi-soliton solutions for the focusing energy-critical wave equation in five-dimensional space:
\begin{equation}\label{wave}
\left\{ \begin{aligned}
&\partial_t^2 u - \Delta u - |u|^{\frac 4{3}} u = 0, \quad (t,x)\in [0,\infty)\times \mathbb{R}^5,\\
& u_{|t=0} = u_0\in \dot H^1,\quad 
\partial_t u_{|t=0} = u_1\in L^2.
\end{aligned}\right.
\end{equation}
Recall that the Cauchy problem for equation \eqref{wave} is locally well-posed in the energy space $\dot H^1\times L^2$, using suitable Strichartz estimates. See \emph{e.g.}~\cite{KM} and references therein.  In particular, for any initial data $(u_{0} , u_{1})\in \dot{H}^{1}\times L^{2}$, there exists a maximal interval of existence
$(T_{-} ,T_{+} )$, $-\infty\le  T_{-} <0< T_{+}\le+\infty$, and a unique solution $(u,\partial_{t} u)\in C
\left(\left(T_{-} ,T_{+}\right);\dot{H}^{1}\times L^{2}\left(\mathbb{R}^{5}\right)\right)\cap L_{\mathrm{loc}}^{\frac{7}{3}}\left( (T_{-} ,T_{+}) ; L^{\frac{14}{3}}(\mathbb{R}^{5})\right)$. For  $\dot H^1\times L^2$ solution, the energy 
$E(u(t),\partial_t u(t))$ and momentum $M(u(t),\partial_t u(t))$ are conserved, where
$$
E(u,v) = \frac 12 \int v^2 + \frac 12 \int |\nabla u|^2 - \frac {3}{10} \int |u|^{\frac {10}{3}},
\quad
M(u,v)=\int v\nabla u.
$$
For a function $u: \mathbb{R}^{5}\to \mathbb{R}$ and $\lambda>0$, we denote
\begin{equation*}
u_{\lambda}(x)=\frac{1}{\lambda^{\frac{3}{2}}}u(\frac{x}{\lambda}).
\end{equation*}
A change of variables shows that
\begin{equation*}
E((u_{0})_{\lambda},\lambda^{-1}(u_{1})_{{\lambda}})= E(u_{0},u_{1} ).
\end{equation*}
Equation~\eqref{wave} is called \emph{energy-critical} since it is invariant under the same scaling: if $(u,\partial_{t}u)$ is a solution of~\eqref{wave} and $\lambda>0$, then
\begin{equation*}
(t,x)\to \left(\frac{1}{\lambda^{\frac{3}{2}}}u\left(\frac{t}{\lambda},\frac{x}{\lambda}\right),\frac{1}{\lambda^{\frac{5}{2}}}\partial_{t}u\left(\frac{t}{\lambda},\frac{x}{\lambda}\right)\right)
\end{equation*}
is also a solution with initial data $\left((u_{
0})_{\lambda},\lambda^{-1}(u_{1} )_{{\lambda}}\right)$
at time $t=0$.

Recall that the function $W$ defined by
\begin{equation}\label{defW}
W(x) = \left( 1+ \frac {|x|^2}{15}\right)^{-\frac{3}2},\quad 
\Delta W + W^{\frac 73}=0 , \quad x\in {\mathbb{R}}^5,
\end{equation}
is a stationary solution of~\eqref{wave}, called here \emph{ground state}, or \emph{soliton}.
By scaling, translation invariances and change of sign, we obtain a family of stationary solutions of~\eqref{wave} defined by $W_{\lambda,x_0,\pm}(x)=\pm \lambda^{-\frac 32}W\left(\lambda^{-1}(x-x_0)\right)$, where $\lambda>0$ and $x_0\in {\mathbb{R}}^5$.

Using the Lorentz transformation, we obtain \emph{traveling waves}.
For $\boldsymbol{\ell}\in {\mathbb{R}}^5$, with $|\boldsymbol{\ell}|< 1$, let
\begin{equation}\label{defWbb}
W_{{\boldsymbol{\ell}}}(x)=W\left(\left(\frac{1}{\sqrt{1-|\boldsymbol{\ell}|^2}}-1\right) \frac{\boldsymbol{\ell}(\boldsymbol{\ell}\cdot x)}{|\boldsymbol\ell|^2} +x\right);
\end{equation}
then $u(t,x)=\pm W_{\bell}(x-\bell t)$ is solution of~\eqref{wave}.

In this paper, we prove the existence of solutions of~\eqref{wave} which display non trivial asymptotic behavior in the nonradial case. Indeed, multi-solitons are canonical objects behaving exactly as the sum of decoupled traveling solitons  as $t\to +\infty$.

\begin{theorem}[Existence of multi-solitons]\label{th:1}
	Let $K\ge 2$. For $k\in \{1,\dots,K\}$, let  $\lambda_k^\infty>0,$ ${\mathbf y}_k^\infty\in {\mathbb{R}}^5$, $\epsilon_k\in \{\pm1\}$, ${\boldsymbol{\ell}}_k=\sum_{i=1}^{5}\ell_{k,i}\mathbf{e}_{i}$ with $|{\boldsymbol{\ell}}_k|<1$.
	Assume that ${\boldsymbol{\ell}}_k\neq {\boldsymbol{\ell}}_{k'}$ for $k\ne k'$, and
	\begin{equation}\label{sc}
\left(\sum_{i=1}^{5}\left(\max_{k\in\{1,\ldots,K\}}\ell^{2}_{k,i}\right)\right)^{\frac{1}{2}}< \frac{3}{5}.
	\end{equation}
	Then, there exist  $T_0>0$ and a solution $u$ of \eqref{wave}   on $[T_0,+\infty)$ in the energy space such that
	\begin{align}\label{eq:th1}
	&\lim_{t\to +\infty} \left\|u(t) 
	-\sum_{k=1}^K \frac{\epsilon_k}{(\lambda_k^\infty)^{3/2}} W_{\boldsymbol{\ell}_k}\left(\frac {\cdot -{\boldsymbol{\ell}_k} t-\mathbf{y}_k^\infty} {\lambda_k^\infty}  \right)
	\right\|_{\dot H^1}  =0,
	\\
	\label{eq:th1bis}
	&  \lim_{t\to +\infty}  \left\|\partial_t u(t) 
	+\sum_{k=1}^K \frac{\epsilon_k}{(\lambda_k^\infty)^{5/2}}\, ({\bell_k}\cdot  \nabla W_{\boldsymbol{\ell}_k})\left(\frac {\cdot-{\boldsymbol{\ell}_k} t-\mathbf{y}_k^\infty} {\lambda_k^\infty}  \right)
	\right\|_{L^2} =0.
	\end{align}
\end{theorem}

The construction of multi-solitons for non integrable dispersive and wave equations has been the subject of several previous works. First, the existence of multi-solitons for the nonlinear Schr\"odinger (NLS)  and the generalized Korteweg-de Vries (gKdV) equations was  studied in the mass critical and subcritical cases by Merle~\cite{Mnls}, Martel~\cite{Ma}, and Martel and Merle~\cite{MMnls}.  The proofs are partly based on standard stability properties of solitary waves (see references in these articles). A key point in~\cite{Ma} and~\cite{MMnls} is the introduction of localized versions of the energy and mass functionals to deal with solutions containing several solitons. Such functionals are reminiscent of Martel, Merle and Tsai~\cite{MMT}, where the stability of the sum of several solitons was studied in the energy space.  For  (gKdV), the technique also allows to prove the uniqueness of multi-solitons~\cite{Ma}, property which is not available for  (NLS). Next, the strategy of these works was extended to the case of exponentially unstable solitons: see C\^ote, Martel and Merle~\cite{CMM} for the construction of multi-solitons for supercritical (gKdV) and (NLS), and Combet~\cite{Com} for a classification result for supercritical (gKdV). In these papers, the exponential instability of each soliton is controlled through a simple topological argument.
For the nonlinear Klein-Gordon equation, the strategy was adapted by C\^ote and Mu\~noz \cite{CMkg} (for real-valued, unstable solitons) and Bellazzini, Ghimenti and Le Coz \cite{BGLC} (for complex-valued, stable solitons). For the water-wave system, see the   work of Ming, Rousset and Tzvetkov~\cite{MRT}.

Closely related to the present work, for the energy-critical wave equation in dimension~$5$, Martel and Merle~\cite{MMwave1} proved the existence of  $K$-solitons for any $K\ge 2$ but only in the case where the  speeds $\{\boldsymbol{\ell}_k\}_{k\in\{1,\cdots,k\}}$ are collinear. 
The Lorentz transform allows to remove this assumption  in the special case  $K= 2$.
Technically, such constructions for the energy critical wave equation are more challenging than for  (gKdV)   and  (NLS) (for which no limitation on the speeds appears)  because of the weak decay of the solitons~\eqref{defW}. Specific localized energy-momentum functionals are introduced to control solitons interactions in this context. In~\cite{MMwave1}, the nature of the localization does not allow non-collinear multi-solitons, even under a smallness condition. The main improvements in the present work are the following: (1) the use of a refined approximate solution from~\cite{MMwave2}; (2) the introduction of a new localized energy functional to deal with non-collinear solitons. The smallness condition~\eqref{sc} is still probably technical and we conjecture that the result holds for any choice of speeds. However, by our method,~\eqref{sc} seems to be the best condition using the refined approximate solution from~\cite{MMwave2}. Moreover, constructing  approximate solutions improving the one in~\cite{MMwave2} seems a difficult task.

The construction of $K$-solitons for the energy critical wave equation is especially motivated by the soliton resolution conjecture, fully solved in the $3$D radial case by Duyckaerts,~Kenig and~Merle~\cite{DKMCJM}, and for a subsequence of time in 3, 4 and 5D by Duyckaerts,~Jia,~Kenig and~Merle~\cite{DJKM}. Indeed, the existence  of such multi-solitons complements the statements in~\cite{DJKM,DKMCJM} by exhibiting examples of solutions containing $K$ solitons as $t\to +\infty$ for any $K\ge 2$. In the same direction, recall that for the energy-critical wave equation in dimension $6$, Jendrej~\cite{JJ} proved the existence of radial two-bubble solutions.

\subsection{Outline of the proof}
The strategy of the proof of Theorem 1 combines the use of a refined approximate solution for the $K$-soliton problem from~\cite{MMwave2},   with uniform estimates and compactness, as in~\cite{MMwave1}. See also references in~\cite{MMwave1, MMwave2} for earlier works.

First, we introduce the refined approximate solution of the form
$\vec {\mathbf W} = \sum_{k=1}^{K} (\vec W_k+c_{k}\vec v_k )$ for large $t>0$, where for any $k\in\{1,\dots,K\}$, $\vec W_k$ is a soliton with two time-dependent parameters regarding scaling and translation respectively, and
$\vec v_k$ is a correction term \cite{MMwave2} improving the simpler approximate solution used in~\cite{MMwave1}.
These correction terms of size $t^{-2}$ in the energy space are solutions of non-homogeneous wave equations whose source terms are the main order of the nonlinear interactions of size $t^{-3}$ between the $K$-solitons.
In this way, $\vec {\mathbf W}$ is an approximate solution of the $K$-solitons problem at order $t^{-4}$.
This is decisive for the construction under the smallness condition~\eqref{sc}.

Let $S_{n}\to +\infty$ as $n\to +\infty$ and, for each $n$, let $u_{n}$ be the backwards solution of~\eqref{wave} with initial data at time $S_{n}$
\begin{equation}\label{simb}
\vec{u}_{n}(S_{n},x)\sim \sum_{k=1}^{K}\left(\vec{ W}^{\infty}_{k}\left(S_{n},x\right)+c_{k}\vec{V}^{\infty}_{k}\left(S_{n},x\right)\right)
\end{equation}
(See~\eqref{defun} for a precise definition of $\vec{u}_{n}(S_{n})$). The goal is to prove that there exists a time $T_{0}$ independent of $n$, such that the following uniform estimates~\eqref{eq:out} hold on $[T_{0},S_{n}]$,
	\begin{equation}\label{eq:out}
\left\|\vec{u}_{n}(t) 
-\sum_{k=1}^K \vec{ W}^{\infty}_{k}(t)-\sum_{k=1}^{K}c_{k}\vec{V}^{\infty}_{k}(t)
\right\|_{\dot H^1\times L^{2}}\lesssim \frac{1}{t}.
\end{equation}
The existence of a multi-soliton then follows easily from standard compactness arguments (note that $\vec{u}_{n}(T_{0})$ converges strongly in the energy space $\dot{H}^{1}\times L^{2}$, see \S 5). Thus, we now focus on the proof of~\eqref{eq:out}.

We introduce
\begin{equation*}
\vec \varepsilon=\left(\begin{array}{c}\varepsilon \\ \eta \end{array}\right),\quad 
\vec u_{n} = \left(\begin{array}{c} u_{n} \\ \partial_t u_{n}\end{array}\right) =
\vec {\mathbf W} + \vec \varepsilon, 
\end{equation*}
where
\begin{equation*}
\vec{\mathbf W} = \left(\begin{array}{c} {\mathbf W}
\\[.2cm] {\mathbf X} \end{array}\right) =
\sum_{k=1}^{K} \left( \vec W_k + c_k\vec v_k \right) .
\end{equation*}
By a standard procedure, we choose modulation parameters $\lambda_{k}(t)$ and $\mathbf{y}_{k}(t)$ close to $\lambda_{k}^{\infty}$ and $\mathbf{y}_{k}^{\infty}$, and obtain suitable orthogonality conditions on $\vec{\varepsilon}$. The equation of $\vec{\varepsilon}$ is thus coupled with equations of $\lambda_{k}(t)$ and $\mathbf{y}_{k}(t)$. See Lemma~\ref{le:4}.

To prove~\eqref{eq:out}, we introduce the following energy functional
\[\mathcal H =\int \left\{|\nabla\varepsilon|^2+|\eta|^2- 2 (F ({\mathbf W} +\varepsilon )-F ({\mathbf W} )-f ({\mathbf W} )\varepsilon ) \right\}
+ 2 \int (\vec{\chi}\cdot\nabla \varepsilon )\eta
\] 
where the bounded function $\vec{\chi}$ is equal to $\bell_{k}$ in a neighborhood of the soliton $W_{k}$ and close to $\frac{x}{t}$ in ``transition regions'' between $K$ solitons (see \S  $\mathrm{5.3}$ for a precise definition). The specific choice of function $\vec{\chi}$ is the main novelty of this paper compared to~\cite{MMwave1,MMwave2}.

The functional $\mathcal{H}(t)$ has the following two important properties (see Lemma~\ref{mainprop} for more precise statements):
\\(1) $\mathcal{H}$ is coercive, in the sense that (up to unstable directions, to be controlled separately) it controls the size of $\vec{\varepsilon}$ in the energy space
\begin{equation*}
\mathcal{H}\sim \|\varepsilon\|_{\dot{H}^{1}}^{2}+\|\eta\|_{L^{2}}^{2}.
\end{equation*}
\\(2) The variation of $\mathcal{H}(t)$ is controlled on $[T_{0},S_{n}]$ in the following sense, for any $\delta>0$ small enough,
\begin{equation}\label{outv}
- \frac d{dt} \left( t^{6-3\delta} \mathcal H\right) \lesssim \frac{1}{t^{1+\delta}}.
\end{equation}
Therefore, integrating~\eqref{outv} on $[t, S_{n}]$, from~\eqref{simb}, we find the uniform bound, for any $t\in [T_{0},S_{n}]$, 
\begin{equation}\label{Energy}
\|\varepsilon\|_{\dot{H}^{1}}^{2}+\|\eta\|_{L^{2}}^{2}\lesssim t^{-3+\delta}.
\end{equation}
Using time integration of the equation of the parameters, we can obtain~\eqref{out2} from the above estimate
\begin{equation}\label{out2}
|\lambda_{k}(t)-\lambda_{k}^{\infty}|\lesssim \frac{1}{t},\quad |\mathbf{y}_{k}(t)-\mathbf{y}_{k}^{\infty}|\lesssim \frac{1}{t}.
\end{equation}
Then~\eqref{eq:out} follows from~\eqref{Energy},~\eqref{out2}.
\subsection*{Acknowledgements}
I would like to give the most sincere thanks to Prof. Yvan Martel for his support and his direction during the writing of this article. I am also grateful to the anonymous referees for careful reading and useful suggestions.
\section{Notation and preliminaries}
\subsection{Notation}
We denote
\begin{equation*}
 (g,{\tilde g})_{L^{2}}   =\int g \tilde g,\quad  \|g\|_{L^{2}}^2 = \int |g|^2,\quad  (g, {\tilde g})_{\dot{H}^{1}} =\int \nabla g \cdot \nabla {\tilde g} ,
\quad  \|g\|_{\dot{H}^{1}}^2=\int |\nabla g|^2.
\end{equation*}
For 
$$\vec{g} = \left(\begin{array}{c}g \\h\end{array}\right),
\ \vec{\tilde g} = \left(\begin{array}{c}\tilde g \\\tilde h\end{array}\right),$$
set
\begin{align*}
({\vec{g}}, {\vec{\tilde g}})_{L^{2}}   = (g ,{\tilde g})_{L^{2}} + (h,{\tilde h})_{L^{2}},\quad
({\vec {g}}, {\vec{\tilde g}})_{\dot H^1\times L^2}  =  (g ,{\tilde g})_{\dot{H}^{1}} +  (h,{\tilde h})_{L^{2}}
.
\end{align*}

Set $\langle x \rangle = (1+|x|^2)^{\frac 12}$. Let
\begin{equation*}
\label{aL}
\Lambda = \frac 32 + x \cdot \nabla ,
\quad \widetilde \Lambda = \frac 5 2 + x \cdot \nabla ,\quad
\widetilde\Lambda \nabla=\nabla\Lambda,\quad 
\vec \Lambda = \left(\begin{array}{c}\widetilde \Lambda \\ \Lambda\end{array}\right).
\end{equation*}
Let
$$
J=\left(\begin{array}{cc}0 & 1 \\-1 & 0\end{array}\right).
$$
For $\bell\in \mathbb{R}^{5}$ with $|\bell|<1$, denote
\begin{equation*}
(g,\tilde{g})_{\dot{H}_{\bell}^{1}}=\int \big(\nabla g\cdot\nabla \tilde{g}-(\bell\cdot\nabla g)(\bell\cdot\nabla \tilde{g}),\quad \|g\|_{\dot{H}_{\bell}^{1}}=\|g\|_{\dot{H}^{1}}^{2}-\|\bell\cdot\nabla g\|_{L^{2}}
\end{equation*}
Note that if we define
\begin{equation*}
g_{\bell}(x)=g\left(\left(\frac{1}{\sqrt{1-|\bell|^{2}}}-1\right)\frac{\bell(\bell\cdot x)}{|\bell|^{2}}+x\right)
\end{equation*}
and similarly for $\tilde{g}$, $\tilde{g_{\bell}}$, 
then
\begin{equation*}
({g_{\bell}},{\tilde{g_{\bell}}})_{\dot{H}_{\bell}^{1}}= (1-|\bell|^ 2)^{\frac 12} (g,{\tilde g})_{\dot{H}^{1}}.
\end{equation*}
Set
\[
x_{\bell} = \left( \frac{1}{\sqrt{1-|\bell|^{2}}}-1\right)\frac{\bell(\bell\cdot x)}{|\bell|^{2}}+x-\frac{\bell t}{\sqrt{1-|\bell|^{2}}},\quad
A_{\bell} = \partial_t + \bell\cdot\nabla,
\]
\[B_{\bell} = \partial_t^2 - (\bell\cdot\nabla)\cdot(\bell\cdot\nabla) = A_{\bell}^2 - 2 (\bell\cdot\nabla) A_{\bell},
\]
\[\Lambda_{\bell} = \frac 32 + (x-\bell t) \cdot \nabla,\quad
\Delta_{\bell} = \Delta-(\bell\cdot\nabla)(\bell\cdot\nabla).
\]
For $\alpha>0$ small to be fixed later, set
\begin{equation}\label{phia}
\varphi_\alpha (x) = (1+|x|^2)^{- \alpha}.
\end{equation}

We recall standard Sobolev and H\"older inequalities
\begin{equation}\label{sobolev}
\|u\|_{L^{10/3}}\lesssim \|u\|_{\dot H^1},\quad \|u\|_{L^{10}} \lesssim \|u\|_{\dot H^2},
\end{equation}
\begin{equation}\label{holder1}
\int |u| |v| |w|^{\frac 43} \lesssim \|u\|_{L^{10/3}}\|v\|_{L^{10/3}}\|w\|_{L^{10/3}}^{4/3}\lesssim
\|u\|_{\dot H^1}\|v\|_{\dot H^1}\|w\|_{\dot H^1}^{4/3},
\end{equation}
\begin{equation}\label{lor}
\|uv\|_{L^{10/7}}\lesssim \|u\|_{L^{10/3}}\|u\|_{L^{5/2}},\quad 
\|uvw\|_{L^{10/7}}\lesssim \|u\|_{L^{10/3}}\|v\|_{L^{10/3}}\|w\|_{L^{10}}.
\end{equation}

Denote
\begin{equation*}
f(u)=|u|^{\frac{4}{3}}u,\quad F(u)=\frac{3}{10}|u|^{\frac{10}{3}}.
\end{equation*}
\subsection{Energy linearization around $W$}
Let
\[
L = -\Delta - \frac 73 W^{\frac 43} ,\quad 
(L g,g) = \int \left(|\nabla g|^2 - \frac 73 W^{\frac 43} g^2\right),
\]
\[ H = \left(\begin{array}{cc} L & 0 \\0 & {\rm Id}\end{array}\right),\quad 
(H \vec g,\vec g) = (L g,g)+ \|h\|_{L^2}^2 \quad \text{for} \quad\vec g = \left(\begin{array}{c}g \\h\end{array}\right).
\]
For $\vec g$ small in the energy space, we recall the expansion of the energy
\begin{equation}\label{enerlin}
	E(W+g,h)=E(W,0)+\frac{1}{2}(H\vec{g},\vec{g})_{L^{2}}+ O(\|g\|_{\dot{H}^{1}}^3).
\end{equation}
We recall some technical properties of the operator $L$.
\begin{lemma}[\cite{MMwave1}]\label{l:Q}
	\emph{(i) Spectrum.} The operator $L$ on $L^2$ with domain $H^2$ is a self-adjoint operator with essential spectrum $[0,+\infty)$, no positive eigenvalue and only one negative eigenvalue $-\lambda_0$, with a smooth radial positive eigenfunction $Y \in \mathcal S(\mathbb{R}^5)$.
	Moreover,
	\begin{equation*}
	L (\Lambda W) = L (\partial_{x_j} W) =0, \quad \hbox{for any $j=1,\ldots,5$.}
	\end{equation*}
	
	\emph{(ii) Localized coercivity.} For $\alpha>0$ small enough, there exists $\mu>0$ such that, for all $g \in \dot H^1$, the following holds.
	\begin{align*}
	& \int |\nabla g|^2 \varphi^2- f'(W) g^2 
	\geq   \mu  \int |\nabla g|^2 \varphi^2 
	-\frac 1{\mu} \left( ( g,{\Lambda W})_{\dot{H}^{1}}^2 + \sum_{j=1}^5 (g,{\partial_{x_j} W})_{\dot{H}^{1}}^2+( g,{Y})_{L^{2}}^2\right).
	\label{stloc}
	\end{align*}
\end{lemma}
\subsection{Energy linearization around $W_{\bell}$}\label{sec.23}
For $\bell \in \mathbb{R}^5$ with 
$|\boldsymbol{\ell}|< 1$, $W_{\bell}$ defined in \eqref{defWbb} solves
\begin{equation}\label{eqWbb}
\Delta W_{\bell}-{\bell}\cdot\nabla({\bell}\cdot\nabla W_{\bell})+W_{\bell}^{\frac 73}=0.
\end{equation}
so that $u(t,x) = W_{\bell}\left(x-\bell t\right)$ is a solution of \eqref{wave}. 
Note that
\begin{equation}\label{EWb}
E(W_{\bell},-\bell\cdot  \nabla W_{\bell})-\int|\bell\cdot\nabla W_{\bell}|^2
=   (1-|\bell|^2)^{\frac 12} E(W,0).
\end{equation}
The following operators are related to the linearization of the energy around $W_{\bell}$. Let
\begin{align*}
&L_{\bell}   = -\Delta +{\bell}\cdot\nabla({\bell}\cdot\nabla)  - f'(W_{\bell})  ,\\
&(L_{\bell}g,g)_{L^{2}}=\int \nabla g\cdot \nabla g-(\bell\cdot \nabla g)\cdot(\bell\cdot \nabla g)-f'(W_{\bell})g^{2},\\
&H_{\bell} = \left(\begin{array}{cc}  -\Delta - f'(W_{\bell}) & -{\bell}\cdot \nabla \\ {\bell} \cdot \nabla & {\rm Id}\end{array}\right),\quad (H_{\bell}\vec{g},\vec{g})_{L^{2}}=(L_{\bell}g,g)_{L^{2}}+\|\bell\cdot\nabla g+h\|_{L^{2}}^{2}
\end{align*}
Indeed, proceeding as in \eqref{enerlin}, using \eqref{eqWbb} and \eqref{EWb}, we obtain
\begin{align*}
&E(W_{\bell} + g, -\bell\cdot\nabla W_{\bell}+h) +\int \bell\cdot\nabla(W_{\bell}+g)(-\bell\cdot\nabla W_{\bell}+h) \\
&=(1-|\bell|^{2})^{\frac{1}{2}}E(W,0)+\frac{1}{2}(H_{\bell}\vec{g},\vec{g})^{2}_{L^{2}}+O(\|g\|_{\dot{H}^{1}}^{3}).
\end{align*}
Set  
\begin{align*}
\vec{Z}_{\bell}^\Lambda = \left(\begin{array}{c} \Lambda W_{\bell} \\ - {\bell}\cdot \nabla(\Lambda W_{\bell})\end{array}\right),\quad
\vec{Z}_{\bell}^{\nabla_j} = \left(\begin{array}{c} \partial_{x_j} W_{\bell} \\ - {\bell}\cdot \nabla (\partial_{x_{j}} W_{\bell})\end{array}\right),\quad
\vec{Z}_{\bell}^W = \left(\begin{array}{c} W_{\bell} \\ - {\bell}\cdot \nabla W_{\bell} \end{array}\right),
\end{align*}
$$
Y_{\bell}=Y\left(\left(\frac{1}{\sqrt{1-|\boldsymbol{\ell}|^2}}-1\right) \frac{\boldsymbol{\ell}(\boldsymbol{\ell}\cdot x)}{|\boldsymbol\ell|^2} +x\right),
\quad 
\vec{Z}_{\bell}^{\pm} = \left(\begin{array}{c} \left({\bell} \cdot \nabla  Y_{\bell} 
\pm \frac {\sqrt{\lambda_0}}{\sqrt{1-|\bell|^2}} Y_{\bell}\right) e^{\pm \frac {  \sqrt{\lambda_0}}{\sqrt{1-|{\bell}|^2}} {\bell} \cdot x} \\ 
Y_{\bell}  e^{\pm \frac { \sqrt{\lambda_0}}{\sqrt{1-|{\bell}|^2}} {\bell} \cdot  x } \end{array}\right).$$

We recall following several technical facts.

\begin{lemma}[\cite{MMwave2}]\label{surZZ} For $\bell\in \mathbb{R}^{5}$ with $|\bell|<1$,
	\begin{enumerate}[label=\emph{(\roman*)}]
		\item\emph{Properties of $L_{\boldsymbol{\ell}}$.}
		\[
		L_{\boldsymbol{\ell}} (\Lambda W_{\boldsymbol{\ell}}) = L_{\boldsymbol{\ell}} (\partial_{x_j} W_{\boldsymbol\ell})=0,\quad L_{\boldsymbol{\ell}} Y_{\boldsymbol{\ell}} = -\lambda_0 Y_{\boldsymbol{\ell}},\quad 
		L_{\boldsymbol{\ell}} W_{\boldsymbol{\ell}} = - \frac 43 W_{\boldsymbol{\ell}}^{\frac 73}.
		\]
		\item\emph{Properties of $H_{\boldsymbol{\ell}}$ and $H_{\boldsymbol{\ell}} J$.}
		\begin{align*}
		&H_{\boldsymbol\ell} \vec Z_{\boldsymbol\ell}^\Lambda= H_{\boldsymbol\ell} \vec Z_{\boldsymbol\ell}^{\nabla_j}=0,\quad
		H_{\boldsymbol\ell} \vec Z_{\boldsymbol\ell}^W= - \frac 43 \left(\begin{array}{c} W_{\boldsymbol\ell}^{\frac 73} \\ 0 \end{array}\right),
		\\
		&\psl {H_{\boldsymbol\ell} \vec Z_{\boldsymbol\ell}^W}{\vec Z_{\boldsymbol\ell}^W} = -\frac 43 \int W_{\boldsymbol\ell}^{\frac {10}3},
		\quad
		- H_{\boldsymbol\ell} J (\vec Z_{\boldsymbol\ell}^\pm) = \pm \sqrt{\lambda_0} (1-|{\boldsymbol\ell}|^2)^{\frac 12} \vec Z_{\boldsymbol\ell}^\pm,
		\\
		&\left(\vec Z_{\boldsymbol\ell}^\Lambda,\vec Z_{\boldsymbol\ell}^W\right)_{\dot H_{\bell}^1\times L^2}=
		\left(\vec Z_{\boldsymbol\ell}^{\nabla_j},\vec Z_{\boldsymbol\ell}^W\right)_{\dot H_{\bell}^1\times L^2}=0,
		\quad \psl { \vec Z_{\boldsymbol\ell}^\Lambda}{\vec Z_{\boldsymbol\ell}^\pm}=
		\psl { \vec Z_{\boldsymbol\ell}^{\nabla_j}}{\vec Z_{\boldsymbol\ell}^\pm}=0.
		\end{align*}
		\item\emph{Coercivity.}
		For $\alpha>0$ small enough, there exists $\mu>0$ such that, for all $\vec g \in \dot H^1\times L^2$,
		\begin{multline*}
		\int \left(|\nabla g|^2 \varphi_\alpha^2 -\frac 73 W_{\bell}^{\frac 43} g^2 + h^2 \varphi_\alpha^2
		+ 2 ({\bell}\cdot\nabla g)h \varphi_\alpha^2 \right)
		\\ \geq\mu\int\left(|\nabla g|^2+h^2\right)\varphi_\alpha^2 
		- \frac 1{\mu }\left\{ ({g},{\Lambda W_{\bell}})_{\dot{H}_{\bell}^{1}}^2+|({g},{\nabla W_{\bell}})_{\dot{H}_{\bell}^{1}}|^2
		+ ({\vec g},{\vec Z_{\bell}^{+}})_{L^{2}}^2 + ({\vec g},{\vec Z_{\bell}^{-}})_{L^{2}}^2\right\}.
		\end{multline*}
	\end{enumerate}
\end{lemma}
\subsection{Approximate solution to a non-homogeneous linearized equation}
Let $|\bell| <1$ and $F$, $G$ be defined by
\begin{equation*}
F = W^{\frac 43} + \kappa_{\bell} \Lambda W , \quad 
G = \kappa_{\bell}(1-|\bell|^2)^{-\frac 12} \bell\cdot\nabla \Lambda W ,\quad \kappa_{\bell} = - (1-|\bell|^2) \frac{(W^{\frac 43},\Lambda W)}{\|\Lambda W\|_{L^2}^2}>0 .
\end{equation*}
Set
\begin{equation*}
w_{\bell} (t,x) = W ( x_{\bell} ),\quad F_{\bell} (t,x) = F ( {x_{\bell}} ),\quad
G_{\bell} (t,x) = G ( {x_{\bell}} ).
\end{equation*}

We recall following technical lemma in~\cite{MMwave2}. 
\begin{lemma}\label{pr:AS1}
	There exists a smooth function $v_{\bell}$ such that, for all $0<\delta<1$ and all $t\geq 1$, 
	\begin{equation*}
	\|(v_{\bell},\partial_t v_{\bell})(t)\|_{\dot H^1\times L^2}\lesssim t^{-2},\quad
	\|v_{\bell}(t)\|_{L^2}\lesssim t^{-\frac 32+\delta},\quad
	\left\| \mathcal E_{\bell}(t) \right\|_{H^{1}}\lesssim t^{-4+\delta},
	\end{equation*}
	where 
	\[
	\mathcal E_{\bell}=
	\partial_t^2 v_{\bell} - \Delta v_{\bell} - \frac 73 w_{\bell}^{\frac 43} v_{\bell} 
	- f_{\bell} - g_{\bell},\quad f_{\bell} = t^{-3} F_{\bell},\quad g_{\bell} =t^{-2} G_{\bell}.
	\]
	Moreover, for all $m\geq 0$, $|\alpha|= 1$, $|\alpha'|\geq 2$, $t\geq 1$, $x\in {\mathbb{R}}^5$,
	\begin{equation}\label{e:n40}\begin{aligned}
	|A_{\bell}^m v_\ell(t,x)| &\lesssim (t+\langle x_{\bell}\rangle)^{-1} t^{-(1+m)} \langle x_{\bell}\rangle^{-2+\delta},\\
	|A_{\bell}^m\partial^\alpha v_{\bell}(t,x)| &\lesssim (t+\langle x_{\bell}\rangle)^{-1}t^{-(1+m)} \langle x_{\bell}\rangle^{-3+\delta},\\
	|A_{\bell}^m\partial^{\alpha'} v_{\bell}(t,x)| &\lesssim (t+\langle x_{\bell}\rangle)^{-1}t^{-(1+m)} \langle x_{\bell}\rangle^{-4+\delta},
	\end{aligned}
	\end{equation}
\begin{equation}\label{rk33}
|A_{\bell}^m \partial_t v_{\bell}|\lesssim |A_{\bell}^{m+1}v_{\bell}|+|A_{\bell}^m\nabla v_{\bell}|\lesssim t^{-(2+m)}\langle x_{\bell}\rangle^{-3+\delta},
\end{equation}
	and
	\begin{equation}\label{pres1}\begin{aligned}
	& |A_{\bell}^m \mathcal E_{\bell}(t,x)|\lesssim t^{-(4+m)+\delta} \langle x_{\bell}\rangle^{-3},\quad
	|A_{\bell}^m\partial^{\alpha} \mathcal E_{\bell}(t,x)|\lesssim t^{-(4+m)+\delta} \langle x_{\bell} \rangle^{-4}
	,\\
	& |A_{\bell}^m\partial^{\alpha'} \mathcal E_{\bell}(t,x)|\lesssim t^{-(4+m)+\delta} \langle x_{\bell} \rangle^{-5}.
	\end{aligned}
	\end{equation}
\end{lemma} 
\begin{proof}
See proof of Lemma 3.1 in~\cite{MMwave2}. Although this lemma is proved only for $\bell=\ell \mathbf{e}_{1}$, by rotation, we see that it holds for any $\bell \in \mathbb{R}^{5}$ with $|\bell|<1$.
\end{proof}

\section{Refined approximate solution of $K$-solitons problem}

In this section, we construct a refined approximate solution to the $K$-soliton problem. Let $K\ge 2$ and for any $k\in \{1,\dots,K\}$, let
\[
\lambda_k^\infty>0,\quad \mathbf{y}_k^\infty\in {\mathbb{R}}^5,\quad \epsilon_k=\pm1,\quad {\boldsymbol{\ell}}_k\in {\mathbb{R}}^5,
\quad \hbox{where $|{\boldsymbol{\ell}}_k|<1$}.\]

Let $C_0\gg1$ and $T_0\gg 1$ to be fixed and $I\subset [T_0,+\infty)$ be an interval of ${\mathbb{R}}$. We assume that these functions satisfy, for all $t\in I$,
\begin{equation}\label{BS0}
|\lambda_k(t)-\lambda^\infty_k|+|\mathbf{y}_k(t)-\mathbf{y}_k^\infty| \leq C_0 t^{-1}.
\end{equation}
For $k\in\{1,\dots,K\}$,
we consider $\mathcal C^1$ functions $\lambda_k>0$, $\mathbf{y}_k\in {\mathbb{R}}^5$ defined on $I$.
For $\vec G= (G,H)$, define
\[(\theta_k G)(t,x) = \frac {\epsilon_k}{\lambda_k^{\frac 32}(t)} G\left(\frac {x -{\boldsymbol{\ell}}_k t - \mathbf{y}_k(t)}{\lambda_k(t)}\right),
\quad
\vec \theta_k \vec G = \left(\begin{array}{c}\theta_k G \\[.2cm] \displaystyle \frac {\theta_k} {\lambda_k } H \end{array}\right),\quad
\vec {\tilde \theta}_k \vec G = \left(\begin{array}{c}\displaystyle \frac {\theta_k}{\lambda_k} G\\[.4cm] \theta_k H \end{array}\right).
\]
In particular, set
\begin{equation}\label{defWk}
W_k = \theta_k W_{{\boldsymbol{\ell}}_k} ,\quad X_k=-{\boldsymbol{\ell}}_k\cdot \nabla W_k,\quad 
\vec W_k = \left(\begin{array}{c} W_k \\[.2cm] \displaystyle X_k \end{array}\right).
\end{equation}

\subsection{Main interaction terms}
Expanding the nonlinearity $|u|^{\frac 43}u$ at $u=\sum_{k=1}^{K}W_{k}$, we prove that the main order of the nonlinear interactions is equivalent to the form $t^{-3} \sum_{k=1}^{K} c_k |W_k|^{\frac 43}$. The remaining error term is of size $t^{-4}$.
\begin{lemma}\label{le:int}
	Assume~\eqref{BS0}. For $k,k'\in\{1,2,\dots, K\}$, $k\neq k'$, let
	\[
	{\boldsymbol{\sigma}}_{k,k'} = 
	\left(\frac 1{\sqrt{1-|{\boldsymbol{\ell}}_{k'}|^2}} - 1\right) \frac{{\boldsymbol{\ell}}_{k'} ({\boldsymbol{\ell}}_{k'}\cdot({\boldsymbol{\ell}}_k-{\boldsymbol{\ell}}_{k'}))}{|{\boldsymbol{\ell}}_{k'}|^2}
	+{\boldsymbol{\ell}}_k-{\boldsymbol{\ell}}_{k'},
	\]
	and
	$
	c_{k} = \frac 73(15)^{\frac 32} \sum_{k'\ne k}{\epsilon_{k'}(\lambda_{k'}^\infty)^{\frac 32} }{|{\boldsymbol{\sigma}}_{k,k'}|^{-3}}.
	$
	Then,
	\[
	\left| \sum_{k=1}^{K} W_k\right|^{\frac 43} \left(\sum_{k=1}^{K} W_k\right)- \sum_{k=1}^{K} |W_k|^{\frac 43} W_k =
	t^{-3}\sum_{k=1}^{K} c_k |W_k|^{\frac 43}+{\mathbf R}_{\Sigma },
	\]
	where, for all $t\in I$,
	\begin{equation}\label{NLIN} 
	\left\| {\mathbf R}_{\Sigma }\right\|_{H^1} \lesssim t^{-4}.
	\end{equation}
\end{lemma}

\begin{proof}
The proof is similar to that of Lemma 3.1 in~\cite{MMwave2}. Let $\sigma= \frac 1{10} \min_{k_{1}\ne k_{2}}|{\boldsymbol{\ell}}_{k_{1}}-{\boldsymbol{\ell}}_{k_{2}}|$ and 
$
B_k(t)=\{x,\ |x-{\boldsymbol{\ell}}_k t|\leq \sigma t\}$, $B(t)=\cup_k B_k(t)$.

First, we claim, $\forall k\in\{1,\dots,K\}$ and $k\ne k'$
\begin{equation}\label{inbwpp}
\|W_{k}^{\frac{7}{3}}\|_{L^{2}(B^{c}_{k})}\lesssim t^{-4},\quad \|W_{k}^{\frac{4}{3}}\|_{L^{2}(B^{c}_{k})}\lesssim t^{-\frac{9}{2}}
\end{equation}
and 
\begin{equation}\label{outbwpp}
 \||W_{k}|^{\frac{1}{3}}|W_{k'}|^{2}\|_{L^{2}(B_{k})}\lesssim t^{-\frac{9}{2}},\quad\|W_{k'}^{\frac{7}{3}}\|_{L^{2}(B_{k})}\lesssim t^{-\frac{9}{2}}
\end{equation}
Note that, $\forall x\in B^{c}_{k}$, we have
\begin{equation*}
|W_{k}|^{\frac{7}{3}}\lesssim \langle x-\boldsymbol{\ell}_{k}t \rangle^{-7}\lesssim t^{-4}\langle x-\boldsymbol{\ell}_{k}t \rangle^{-3}, \quad |W_{k}|^{\frac{4}{3}}\lesssim \langle x-\boldsymbol{\ell}_{k}t \rangle^{-4}\lesssim t^{-1}\langle x-\boldsymbol{\ell}_{k}t \rangle^{-3},
\end{equation*}
and from $x\mapsto\langle x \rangle^{-3} \in L^2({\mathbb{R}}^5)$, we obtain \eqref{inbwpp}. Note that,  $\forall x\in B_{k}$ and $k'\ne k$
\begin{equation}\label{outdW}
|W_{k}|^{\frac{1}{3}}|W_{k'}|^{2}\lesssim t^{-6}\langle x-\boldsymbol{\ell}_{k}t \rangle, \quad |W_{k'}|^{\frac{7}{3}}\lesssim t^{-\frac{9}{2}}
\end{equation}
Integrating \eqref{outdW} on $B_{k}(t)$, we obtain \eqref{outbwpp}.

Second, we claim, $\forall\ k\in\{1,\dots,K\}$,
\begin{equation}\label{taylor int}
\||\sum_{k'=1}^{K}W_{k'}|^{\frac{4}{3}}(\sum_{k'=1}^{K}W_{k'})-|W_{k}|^{\frac{4}{3}}W_{k}-\frac{7}{3}|W_{k}|^{\frac{4}{3}}\sum_{k'\ne k}W_{k'}\|_{L^{2}(B_{k})}\lesssim t^{-\frac{9}{2}}
\end{equation}
From Taylor expansion, we have,
\begin{equation}\label{taylor p.p}
\begin{aligned}
&\left|\sum_{k'=1}^{K}W_{k'}\right|^{\frac{4}{3}}(\sum_{k'=1}^{K}W_{k'})\\
&=|W_{k}|^{\frac{4}{3}}W_{k}+\frac{7}{3}|W_{k}|^{\frac{4}{3}}\sum_{k'\ne k}W_{k'}+O\left(|W_k|^{\frac 13} \sum_{k'\ne k}|W_{k'}|^2\right)+O\left(\sum_{k'\ne k}|W_{k'}|^{\frac 73}\right).
\end{aligned}
\end{equation}
Using \eqref{outbwpp} and \eqref{taylor p.p}, we obtain \eqref{taylor int}.

Third, we claim, $k'\neq k$,
\begin{equation}\label{step3}
\left\| |W_k|^{\frac 43} \left(W_{k'}(t,x)-W_{k'}(t,{\boldsymbol{\ell}}_k t) \right) \right\|_{L^2(B_k)} \lesssim t^{-4}.
\end{equation}
Indeed, for $x\in B_k$,
\[\left|W_{k'}(t,x)-W_{k'}(t,{\boldsymbol{\ell}}_k t) \right| \lesssim \sup_{B_k} |\nabla W_{k'}(t)| \cdot |x-{\boldsymbol{\ell}}_k t|
\lesssim t^{-4} |x-{\boldsymbol{\ell}}_k t|,
\]
and so,
\[
|W_k|^{\frac 43} \left|W_{k'}(t,x)-W_{k'}(t,{\boldsymbol{\ell}}_k t) \right| \lesssim
t^{-4} \langle x-{\boldsymbol{\ell}}_k t\rangle^{-3},
\]
which implies~\eqref{step3}.

Then, note that, from the explict expression \eqref{defW} of $W$, we have,
\begin{equation}\label{ W asy}
|W(x)-15^{\frac{3}{2}}|x|^{-3}|\lesssim |x|^{-5},\quad \mathrm{for}\ |x|\gg 1.
\end{equation}
Thus, using the assumption on the parameters \eqref{BS0} and the definition \eqref{defWk} of $W_{k'}$ and \eqref{ W asy}, we obtain
\begin{multline*}
W_{k'}(t,{\boldsymbol{\ell}}_kt) = \frac {\epsilon_{k'}}{\lambda_{k'}^{\frac 32}(t)}
W_{\ell_{k'}}\left(\frac{({\boldsymbol{\ell}}_k-{\boldsymbol{\ell}}_{k'})t-\mathbf y_{k'}(t)}{\lambda_{k'}(t)}\right) = \frac {\epsilon_{k'}}{(\lambda_{k'}^\infty)^{\frac 32}}
W_{\ell_{k'}}\left(\frac{({\boldsymbol{\ell}}_k-{\boldsymbol{\ell}}_{k'})t}{\lambda_{k'}^\infty}\right) +O(t^{-4})\\
= \frac {\epsilon_{k'}}{(\lambda_{k'}^\infty)^{\frac 32}}
W \left(\frac{{\boldsymbol{\sigma}}_{k,k'}t}{\lambda_{k'}^\infty}\right) +O(t^{-4})
= {15^{\frac 32}\epsilon_{k'}}{(\lambda_{k'}^\infty)^{\frac 32}}
|{\boldsymbol{\sigma}}_{k,k'}|^{-3} t^{-3} +O(t^{-4}).
\end{multline*}
Gathering above estimates, we obtain the $L^2$ estimate of ${\mathbf R}_{\Sigma}$ in~\eqref{NLIN}.

Now, we observe for $j\in \{1,\cdots,5\}$,
\begin{align*}
\partial_{x_{j}}{\mathbf R}_{\Sigma}
&=\frac{7}{3}\left|\sum_{k=1}^{K}W_{k}\right|^{\frac{4}{3}}\left(\sum_{k=1}^{K}\partial_{x_{j}}W_{k}\right)-\frac{7}{3}\sum_{k=1}^{K}\left(\left|W_{k}\right|^{\frac{4}{3}}\partial_{x_{j}}W_{k}\right)\\
&-\frac{4}{3}t^{-3}\sum_{k=1}^{K}c_{k}\left(\left|W_{k}\right|^{-\frac{2}{3}}W_{k}\right)\partial_{x_{j}}W_{k}.
\end{align*}
From Taylor expansion and the decay of $W$ and $\partial_{x_{j}}W$, we have, for $k\in\{1,\cdots,K\}$,
\begin{align*}
&\left|\sum_{k'=1}^{K}W_{k'}\right|^{\frac{4}{3}}\left(\sum_{k'=1}^{K}\partial_{x_{j}}W_{k'}\right)\\
&=\big|W_{k}\big|^{\frac{4}{3}}\partial_{x_{j}}W_{k}+\frac{4}{3}\left(\big(\big|W_{k}\big|^{-\frac{2}{3}}W_{k}\big)\partial_{x_{j}}W_{k}\right)\big(\sum_{k'\ne k}W_{k'}\big)\\
&+O\big(\big|W_{k}\big|\sum_{k'\ne k}\big|W_{k'}\big|^{\frac{4}{3}}\big)+O\big(\big|W_{k}\big|^{\frac{4}{3}}\sum_{k'\ne k}\big|\partial_{x_{j}}W_{k'}\big|\big)+O\big(\sum_{k'\ne k}\big|W_{k'}\big|^{\frac{7}{3}}\big).
\end{align*}
And then, using similar arguments, we prove $\dot{H}^{1}$ estimate of ${\mathbf R}_{\Sigma}$ in~\eqref{NLIN}.
\end{proof}
\subsection{The approximate solution $\vec{\mathbf W}$}
We recall a result on approximate solution $\vec{\mathbf W}$ in 5D from \cite{MMwave2}. To remove the main interaction terms $\sum_{k=1}^{K}c_k t^{-3} |W_k|^{\frac 43}$ computed in Lemma~\ref{le:int}, we define suitably rescaled versions of the function $v_{\bell_k}$ given by Lemma~\ref{pr:AS1}.
Let
\begin{align}
v_k(t,x) & = \frac 1{ \lambda_k^3} v_{\bell_k}\left( \frac t{ \lambda_k} , \frac {x-\mathbf{y}_k}{ \lambda_k}\right),
\label{def.vk}\\
z_k(t,x) & = \frac 1{ \lambda_k^4} (\partial_tv_{\bell_k})\left( \frac t{ \lambda_k} , \frac {x-\mathbf{y}_k}{ \lambda_k}\right)
+ \frac{\kappa_{\bell_k}\epsilon_k}{2 \lambda_k^{\frac 12}t^2}\Lambda_k W_k(t,x)
\label{def.wk} \end{align}
where $\Lambda_k = \frac 32 + (x -{\boldsymbol{\ell}}_k t -\mathbf{y}_k)\cdot \nabla$,
and
\[
\vec v_k = \left(\begin{array}{c} v_k 
\\[.2cm]z_k \end{array}\right),\quad 
\kappa_{\bell_k} = - (1-|\bell_k|^2) \frac{(W^{\frac 43},\Lambda W)}{\|\Lambda W\|_{L^2}^2},\quad
a_k = - \frac {c_k\kappa_{\bell_k}\epsilon_k}2.
\]
Set
\begin{equation}\label{def:Wapp}
\vec{\mathbf W} = \left(\begin{array}{c} {\mathbf W}
\\[.2cm] {\mathbf X} \end{array}\right) =
\sum_{k=1}^{K} \left( \vec W_k + c_k\vec v_k \right) .
\end{equation}
\begin{lemma}\label{le:43}
 The function $\vec{\mathbf W}$ satisfies on $I\times {\mathbb{R}}^5$ 
	\begin{equation}\label{syst_WX}\left\{\begin{aligned}
	\partial_t{\mathbf W} & = {\mathbf X} - {\rm Mod}_{{\mathbf W}}-{\rm Mod}_{{\mathbf{V}}}\\
	\partial_t	{\mathbf X} & 
	= \Delta {\mathbf W} +|{\mathbf W}|^{\frac 43} {\mathbf W} - {\rm Mod}_{{\mathbf X}}-{\rm Mod}_{{\mathbf{Z}}}- {\mathbf R}_{{\mathbf X}}
	\end{aligned}\right.\end{equation} 
	where
	\begin{align*}
	{\rm Mod}_{{\mathbf W}} &= \sum_{k=1}^{K} \left( \frac{\dot\lambda_k}{\lambda_k} - \frac{a_k}{\lambda_k^{\frac 12}t^2} \right) \Lambda_k W_k
	+ \sum_{k=1}^{K} \dot {\mathbf{y}}_k \cdot \nabla W_k \\
	{\rm Mod}_{{\mathbf X}} &= 
	- \sum_{k=1}^{K} \left( \frac{\dot\lambda_k}{\lambda_k} - \frac{a_k}{\lambda_k^{\frac 12}t^2} \right) ({\boldsymbol{\ell}}_k \cdot \nabla) \Lambda_k W_k 
	- \sum_{k=1}^{K} ({\dot {\mathbf{y}}}_k \cdot \nabla) ( {\boldsymbol{\ell}}_k \cdot \nabla) W_k ,\\
\end{align*}
and
\begin{align*}
	{\rm Mod}_{{\mathbf{V}}} &= \sum_{k=1}^{K}\left(3 \frac{\dot \lambda_k}{\lambda_k^4} v_{{\bell}_k} \left( \frac t{ \lambda_k} , \frac {x-\mathbf{y}_k}{ \lambda_k}\right)
	+ \frac{\dot \lambda_k}{\lambda_k^4} \frac{t}{\lambda_k}A_{{\bell}_k} v_{{\bell}_k} \left( \frac t{ \lambda_k} , \frac {x-\mathbf{y}_k}{ \lambda_k}\right)\right) \\ 
	& \ \ +\sum_{k=1}^{K}\left(\frac{\dot \lambda_k}{\lambda_k^4}\left(\frac{x-{\boldsymbol{\ell}}_k t-\mathbf{y}_k}{\lambda_k}\right)\cdot \nabla v_{{\bell}_k} \left( \frac t{\lambda_k} , \frac {x-\mathbf{y}_k}{ \lambda_k}\right) + \frac{\dot {\mathbf{y}}_k}{\lambda_k^4} \cdot\nabla v_{{\bell}_k} \left( \frac t{ \lambda_k} , \frac {x-\mathbf{y}_k}{ \lambda_k}\right)\right) ,\\
	{\rm Mod}_{{\mathbf{Z}}} &= \sum_{k=1}^{K}\left(\frac{\dot {\mathbf{y}}_k}{\lambda_k^5} \cdot \nabla \partial_t v_{{\bell}_{k}} \left( \frac t{ \lambda_k} , \frac {x-\mathbf{y}_k}{ \lambda_k}\right)+4 \frac{\dot \lambda_k}{\lambda_k^5} \partial_t v_{{\bell}_k}\left( \frac t{ \lambda_k} , \frac {x-\mathbf{y}_k}{ \lambda_k}\right) \right)
	\\ &\quad 
	+ \sum_{k=1}^{K}\left(\frac{\dot \lambda_k}{\lambda_k^5} \frac{t}{\lambda_k} A_{\bell_k} \partial_t v_{\bell_k}\left( \frac t{ \lambda_k} , \frac {x-\mathbf{y}_k}{ \lambda_k}\right) 
	+ \frac{\dot \lambda_k}{\lambda_k^5} \left(\frac{x-{\boldsymbol{\ell}}_{k} t-\mathbf{y}_k}{\lambda_k}\right)\cdot
	\nabla \partial_t v_{\bell_k} \left( \frac t{ \lambda_k} , \frac {x-\mathbf{y}_k}{ \lambda_k}\right)\right) 
	\\&\quad  +\sum_{k=1}^{K}\left(\frac {\kappa_{\bell_k}\epsilon_k}{2t^2\lambda_k^{\frac 12}} \frac{\dot\lambda_k}{\lambda_k}
	\left(\frac 12 \Lambda_k W_k+\Lambda_k^2 W_k\right)
	+\frac {\kappa_{\bell_k}\epsilon_k }{2t^2\lambda_k^{\frac 12}} \dot{\mathbf{y}}_k\cdot\nabla\Lambda_k W_k\right),
\end{align*}
	\begin{align}
	\vec{\mathbf{R}}= \left(\begin{array}{c} 0
	\\[.2cm] {\mathbf R}_{{\mathbf X}}\end{array}\right),\quad 
	\|\vec{\mathbf{R}}\|_{\dot H^1\times L^2} +\|\nabla \vec{\mathbf{R}}\|_{\dot H^1\times L^2}\lesssim t^{-4+\delta}.
	\label{RWX}\end{align}
	Moreover, for all $0<\delta< 1$,
	\begin{equation}\label{e:WX}\begin{aligned}
	&|{\mathbf W}|+\langle x-{\boldsymbol{\ell}}_k t\rangle |\nabla {\mathbf W}|
	\lesssim \sum_{k=1}^{K} \left( \langle x-{\boldsymbol{\ell}}_k t\rangle^{-3}+t^{-1}\langle x-{\boldsymbol{\ell}}_k t\rangle^{-3+\delta}\right),\\
	&|{\mathbf X}|\lesssim \sum_{k=1}^{K} \left( \langle x-{\boldsymbol{\ell}}_k t\rangle^{-4}+t^{-2}\langle x-{\boldsymbol{\ell}}_k t\rangle^{-3+\delta}\right).
	\end{aligned}\end{equation}
\end{lemma}
\begin{proof}
See proof of Lemma 4.3 in~\cite{MMwave2}. Although the original argument in this lemma only holds for ${\bell}_{k}=\ell_{k} \mathbf{e}_{1}$ and $K=2$, after using Lemma 2.3 and Lemma 3.1, we easily check that the argument still holds for any $K\ge2$ and any ${\bell}_{k} \in \mathbb{R}^{5}$ with $|{\bell}_{k}|<1$.
\end{proof}
\section{Decomposition around refined approximate solution}
We prove in this section a general decomposition around refined approximate solution. Let $K\ge 2$ and for any $k\in \{1,\dots,K\}$, let $\lambda_{k}^{\infty}>0$, $\mathbf{y}_{k}^{\infty}\in \mathbb{R}^{5}$, $\bell_{k}\in \mathbb{R}^{5}$, $|\bell_{k}|< 1$ with $\bell_{k}\ne \bell_{k'}$ for $k'\ne k$.
 For $\vec{G}=(G,H)$, set
 \[
 (\theta_k^\infty G)(t,x) = \frac {\epsilon_k}{(\lambda_k^\infty)^{\frac 32}} G\left(\frac {x -{\boldsymbol{\ell}}_k t -\mathbf{y}_k^\infty}{\lambda_k^\infty}\right),
 \quad 
 \vec \theta_k^\infty \vec G 
 = \left(
 \begin{array}{c}\theta_k^\infty G \\[.2cm] 
 \displaystyle \frac {\theta_k^\infty} {\lambda_k^\infty} H \end{array}\right),\]
 \[
 W_k^\infty =\theta_k^\infty W_{{\boldsymbol{\ell}}_k}, \quad X_k^\infty=-{\boldsymbol{\ell}}_k\cdot \nabla W_k^\infty,\quad \vec W_k^\infty =\left(\begin{array}{c} W_k^\infty \\[.2cm] \displaystyle X_k^\infty \end{array}\right),\quad \Lambda^{\infty}_{k}=\frac{3}{2}+(x-\bell_{k}t-\mathbf{y}_{k}^{\infty}).
 \]
 We also set
 \begin{align*}
 &v_{k}^{\infty}=\frac{1}{(\lambda_{k}^{\infty})^{3}}v_{\bell_{k}}\left(\frac{t}{\lambda_{k}^{\infty}},\frac{x-\mathbf{y}^{\infty}_{k}}{\lambda_{k}^{\infty}}\right),\\
 &z_{k}^{\infty}=\frac{1}{(\lambda_{k}^{\infty})^{4}}(\partial_{t}v_{\bell_{k}})\left(\frac{t}{\lambda_{k}^{\infty}},\frac{x-\mathbf{y}^{\infty}_{k}}{\lambda_{k}^{\infty}}\right)+\frac{\kappa_{\bell_k}\epsilon_k}{2 (\lambda^{\infty}_{k})^{\frac 12}t^2}\Lambda^{\infty}_{k} W^{\infty}_{k}(t,x),\quad \vec{V}_{k}^{\infty}=\left(\begin{array}{c} v_{k}^{\infty}\\[.2cm] 
 \displaystyle z_{k}^{\infty}\end{array}\right).
 \end{align*}
\begin{lemma}[Properties of the decomposition]\label{le:4} 
	There exist $T_0\gg 1$ and $0<\delta_0\ll 1$ such that if $u(t)$ is a solution of~\eqref{wave} which satisfies on $I$,
	\begin{equation}\label{hyp:4}
	\left\|\vec u - \sum_{k=1}^{K} \left( \vec W_k^\infty + c_k\vec V_k^{\infty} \right) \right\|_{\dot H^1\times L^2}< \delta_0,
	\end{equation}
	then there exist $\mathcal C^1$ functions $\lambda_k>0$, $\mathbf{y}_k$ on $I$ such that, 
	$\vec \varepsilon(t)$ being defined by
	\begin{equation}
	\label{eps2}
	\vec \varepsilon=\left(\begin{array}{c}\varepsilon \\ \eta \end{array}\right),\quad 
	\vec u = \left(\begin{array}{c} u \\ \partial_t u\end{array}\right) =
	\vec {\mathbf W} + \vec \varepsilon, 
	\end{equation}
	the following hold on $I$, for $k\in\{1,\dots,K\}$.
	\begin{enumerate}[label=\emph{(\roman*)}]
		\item\emph{First properties of the decomposition.} For $j=1,\ldots,5$, 
		\begin{equation}
		\label{ortho}
		\pshbbk {\varepsilon}{ \Lambda_k W_k}=
		\pshbbk {\varepsilon}{ \partial_{x_j} W_k}= 0,
		\end{equation}
		\begin{equation}\label{bounds}
		|\lambda_k -\lambda_k^{\infty}|
		+|\mathbf{y}_k -\mathbf{y}_k^\infty|
		+\|\vec \varepsilon \|_{\dot H^1\times L^2}
		\lesssim \left\|\vec u -\sum_{k=1}^{K} \left( \vec W_k^\infty + c_k\vec V_k^{\infty} \right) \right\|_{\dot H^1\times L^2}
		\end{equation}
		\item\emph{Equation of $\vec \varepsilon$.} 
		\begin{equation}\label{syst_e}\left\{\begin{aligned}
		\partial_t \varepsilon & = \eta + {\rm Mod}_{{\mathbf W}}+{\rm Mod}_{{\mathbf V}}\\
		\partial_t \eta & 
		= \Delta \varepsilon +\left| {\mathbf W} + \varepsilon\right|^{\frac 43} ({\mathbf W} + \varepsilon)
		- |{\mathbf W}|^{\frac 43}{\mathbf W} + {\rm Mod}_{{\mathbf X}}+{\rm Mod}_{{\mathbf Z}}+ {\mathbf R}_{{\mathbf X}} .
		\end{aligned}\right.\end{equation} 
		\item\emph{Parameter estimates.} For any $0<\delta<1$, \begin{equation}
		\label{le:p}
		\left|\frac {\dot \lambda_k }{\lambda_k} - \frac{a_k}{\lambda_k^{\frac 12}t^2}\right|+|\dot {\mathbf{y}}_k|
		\lesssim \left\|\vec \varepsilon\right\|_{\dot H^1\times L^2}+\frac{1}{t^{4}}.
		\end{equation}
		\item\emph{Unstable directions.} Let
		$z_k^{\pm} = (\vec \varepsilon ,\vec {\tilde\theta}_k \vec Z_{{\boldsymbol{\ell}}_k}^{\pm})$.
		Then, for any $0<\delta<1$, 
		\begin{equation}\label{le:z}
		\left| \frac d{dt} z_k^{\pm}\mp\frac{\sqrt{\lambda_0}}{\lambda_k}(1-|{\boldsymbol{\ell}}_k|^2)^{\frac 12}z_k^{\pm} \right|
		\lesssim \left\| {\vec\varepsilon }\right\|_{\dot H^1\times L^2}^2 + t^{-1} \left\|{\vec\varepsilon }\right\|_{\dot H^1\times L^2} + t^{-4+\delta}
		\end{equation}
	\end{enumerate}
\end{lemma}
\begin{proof}
	\textbf{Step 1.} Decomposition. The existence of parameters ${\lambda_k}$ and ${\mathbf y}_k$ such that~\eqref{ortho} and~\eqref{bounds} hold is proved similarly as~(i) of Lemma~3.1 in~\cite{MMwave1}.
	\smallbreak
	
	\textbf{Step 2.} Equation of $\vec \varepsilon$.
We formally derive the equation of $\vec{\varepsilon}(t)$, $\lambda_{k}(t)$ and $\mathbf{y}_{k}(t)$ from~\eqref{wave} and~\eqref{syst_WX}.
\begin{equation*}
\varepsilon_{t}=\partial_{t}u-\partial_{t}\mathbf{W}=\eta+\mathbf{X}-\partial_{t}\mathbf{W}=\eta+{\rm Mod}_{{\mathbf{W}}}+{\rm Mod}_{{\mathbf{V}}}.
\end{equation*}
	Second, since $\eta=\partial_t u-{\mathbf X}$, we have
	\begin{align*}
	\partial_t \eta & = \partial_t ^2 u -\partial_t {\mathbf X} =\Delta u +|u|^{\frac 43} u -\Delta{\mathbf W} -|{\mathbf W}|^{\frac 43} {\mathbf W} +{\rm Mod}_{\mathbf X} +{\rm Mod}_{{\mathbf{Z}}}+ \mathbf{R}_{\mathbf X}\\
	& = \Delta \varepsilon +|{\mathbf W}+\varepsilon|^{\frac 43} ({\mathbf W}+\varepsilon) - |{\mathbf W}|^{\frac 43} {\mathbf W} + {\rm Mod}_{{\mathbf X}} +{\rm Mod}_{{\mathbf{Z}}}+ \mathbf{R}_{\mathbf X}.
	\end{align*}
	We also denote
	\begin{align*}
	{\mathbf R}_{\rm NL} &= \left| {\mathbf W} + \varepsilon\right|^{\frac 43} ({\mathbf W} + \varepsilon)
	- |{\mathbf W}|^{\frac 43}{\mathbf W} -\frac 73 \sum_k |W_k|^{\frac 43}\varepsilon ={\mathbf{R}}_{1}+ {\mathbf{R}}_{2}, \\
	{\mathbf{R}}_{1} &= \frac 73 \left( |{\mathbf W}|^{\frac 43} - \sum_k |W_k|^{\frac 43}\right) \varepsilon,\quad 
	{\mathbf{R}}_{2} = \left| {\mathbf W} + \varepsilon\right|^{\frac 43} ({\mathbf W} + \varepsilon)
	- |{\mathbf W}|^{\frac 43}{\mathbf W} - \frac 73 |{\mathbf W}|^{\frac 43} \varepsilon ,
	\end{align*}
	and
	\[
	\vec {\mathcal L} = \left(\begin{array}{cc}0 & 1 \\
	\Delta + \frac 73 \sum_k |W_k|^{\frac 43} & 0\end{array}\right),\quad
	\vec {\rm Mod}_{1}= \left(\begin{array}{c} {\rm Mod}_{\mathbf W} \\{\rm Mod}_{\mathbf X} \end{array}\right),\quad
	\vec {\rm Mod}_{2}= \left(\begin{array}{c} {\rm Mod}_{\mathbf V} \\{\rm Mod}_{\mathbf Z} \end{array}\right),
	\]
	\[ 
	\vec {{\mathbf R}} = \left(\begin{array}{c} 0 \\{\mathbf R}_{\mathbf X}\end{array}\right),\quad\vec {{\mathbf R}}_{\rm NL} = \left(\begin{array}{c}0 \\{\mathbf R}_{\rm NL}\end{array}\right),\quad
	\vec {{\mathbf R}}_{1} = \left(\begin{array}{c}0 \\{\mathbf R}_{1}\end{array}\right),\quad \vec {{\mathbf R}}_{2} = \left(\begin{array}{c}0 \\{\mathbf R}_{2}\end{array}\right).
	\] 
	With this notation, the system~\eqref{syst_e} rewrites 
	\begin{equation}\label{syst_e3}
	\partial_t \vec \varepsilon = \vec {\mathcal L} \vec \varepsilon + \vec {\rm Mod}_{1}+\vec {\rm Mod}_{2} + \vec {\mathbf{R}}+\vec {\mathbf{R}}_{\rm NL} = \vec {\mathcal L} \vec \varepsilon + \vec {\rm Mod}_{1}+\vec {\rm Mod}_{2} + \vec {\mathbf{R}}+\vec {\mathbf{R}}_1+\vec {\mathbf{R}}_2.
	\end{equation}
	We claim the following estimates on ${\mathbf R}_1$ and ${\mathbf R}_2$
	\begin{equation}\label{R1bis}
	\| {\mathbf{R}}_1\|_{L^{\frac{10}7}}\lesssim 
	t^{-1} \|\varepsilon\|_{L^{\frac{10}3}} ,\quad
	| {\mathbf{R}}_2|\lesssim |{\mathbf W}|^{\frac13}\varepsilon^2 + |\varepsilon|^{\frac 73},\quad
	\| {\mathbf{R}}_2\|_{L^{\frac {10}7}} \lesssim \|\varepsilon\|_{L^{\frac{10}3}}^2 \lesssim \|\varepsilon\|_{\dot H^1}^{2}.
	\end{equation}
	The estimate on $\mathbf R_2$ follows from~\eqref{lor}. To prove the estimate on $\mathbf R_1$, we first recall
	the inequality, for $p>1$, for any reals $(r_k)$,
	\begin{equation*}
	\left|\left|\sum r_k \right|^{p} - \sum |r_k|^p\right| \lesssim \sum_{k'\neq k} |r_{k'}||r_k|^{p-1}.
	\end{equation*}
	Therefore,
	\begin{align*}
	\left| |{\mathbf W}|^{\frac 43} - \sum_{k=1}^{K} |W_k|^{\frac 43}\right|&\lesssim \left| |{\mathbf W}|^{\frac 43} - \left|\sum_{k=1}^{K} W_k\right|^{\frac 43}\right|+\left| \left|\sum_{k=1}^{K} W_k\right|^{\frac 43} - \sum_{k=1}^{K}|W_k|^{\frac 43}\right|\\
	&\lesssim \left(\sum_{k=1}^{K} |v_k|\right) \left( \sum_{k=1}^{K} \left(|W_k| +|v_k|\right)\right)^{\frac 13} + \sum_{k'\neq k} |W_k| |W_{k'}|^{\frac 13},
	\end{align*}
	and thus
	\[
	| {\mathbf{R}}_1|\lesssim 
	|\varepsilon| \left(\sum_{k=1}^{K} |v_k|\right) \left( \sum_{k=1}^{K} \left(|W_k| +|v_k|\right)\right)^{\frac 13} + |\varepsilon|\sum_{k'\neq k} |W_k| |W_{k'}|^{\frac 13}.
	\]
	By~\eqref{lor},
	we obtain
	\[
	\| {\mathbf{R}}_1\|_{L^{\frac{10}7}}\lesssim 
	\|\varepsilon\|_{L^{\frac{10}3}}
	\left( \sum_{k=1}^{K} \|v_k\|_{L^{\frac{10}3}}\right) \left(\sum_{k=1}^{K}\left(\|W_k\|_{L^{\frac{10}3}}^{\frac 13} + \|v_k\|_{L^{\frac{10}3}}^{\frac 13}\right)\right) + \|\varepsilon\|_{L^{\frac{10}3}}\sum_{k'\neq k} 
	\left \|W_k|W_{k'}|^{\frac 13}\right\|_{L^{\frac 52}}.
	\]
	By~\eqref{e:n40}, we have $\|v_k\|_{L^{\frac{10}3}}\lesssim t^{-2}$.
	Moreover $\|W_k|W_{k'}|^{\frac 13} \|_{L^{\frac 52}}\lesssim t^{-1}$ is a consequence of the following technical result.
	\begin{claim}[Claim~2 in~\cite{MMwave1}]\label{WW}
		Let $0<r_2\leq r_1$ be such that $r_1+r_2> \frac 53$. For $t$ large, 
		if $r_1> \frac 53$ then $\int |W_1|^{r_1} |W_2|^{r_2}\lesssim t^{-3 r_2}$, whereas if $r_1\leq \frac 53$ then $\int |W_1|^{r_1} |W_2|^{r_2} \lesssim t^{5-3 (r_1+r_2)}$.
	\end{claim}
	
	\textbf{Step 3.} Parameter estimates.
	Now, we derive the equations of $\lambda_k$ and $\mathbf{y}_k$ from the orthogonality conditions~\eqref{ortho}. First,
	\begin{align*}
	\frac d{dt} \pshbbun {\varepsilon}{ \Lambda_1 W_1 } = 
	\pshbbun{\partial_t \varepsilon}{ \Lambda_1 W_1 }+
	\pshbbun {\varepsilon}{\partial_t \left(\Lambda_1 W_1\right) }=0.
	\end{align*}
	Thus, using the first line of~\eqref{syst_e}, and the expression of ${\rm Mod}_{\mathbf W}$ in Lemma~\ref{le:43},
	\begin{align*}
	0= & \pshbbun{\eta}{ \Lambda_1 W_1 }
	-\pshbbun{\varepsilon}{ {\boldsymbol{\ell}}_1 \cdot \nabla (\Lambda_1 W_1)} -\frac{\dot\lambda_1}{\lambda_1}\pshbbun {\varepsilon}{ \Lambda^2_1W_1 } -\pshbbun {\varepsilon}{{\dot {\mathbf{y}}_1}\cdot \nabla\Lambda_1 W_1} 
	\\
	& +\left( \frac{\dot\lambda_1}{\lambda_1} - \frac{a_1}{\lambda_1^{\frac 12}t^2}\right) \pshbbun{\Lambda_1 W_1 }{ \Lambda_1 W_1 }
	+ \pshbbun{ \dot {\mathbf{y}}_1 \cdot \nabla W_1 }{ \Lambda_1 W_1 } \\
	& + \sum_{k=2}^{K}\left( \frac{\dot\lambda_k}{\lambda_k} - \frac{a_k}{\lambda_k^{\frac 12}t^2}\right) \pshbbun{\Lambda_k W_k}{\Lambda_1 W_1}
	+ \sum_{k=2}^{K}\pshbbun{ \dot {\mathbf{y}}_k \cdot \nabla W_k }{ \Lambda_1 W_1 }+ \pshbbun{{\rm Mod}_{\mathbf{V}}}{ \Lambda_1 W_1 } .
	\end{align*}
	By the decay properties of $W$, we obtain 
	\[
	\left| \pshbbun{\eta}{ \Lambda_1 W_1 }\right|
	+ \left|\pshbbun{\varepsilon}{ \nabla (\Lambda_1 W_1)} \right|
	+\left| \pshbbun {\varepsilon}{ \Lambda_1^2 W_1 }\right| 
	+\left|\pshbbun {\varepsilon}{ \nabla\Lambda_1 W_1}\right|\lesssim \|\vec \varepsilon\|_{\dot H^1\times L^2}.
	\]
	Next, $
	\pshbbun{\Lambda_1 W_1 }{ \Lambda_1 W_1 } 
	= (1-|{\boldsymbol{\ell}}_1|^2)^{\frac 12} \|\Lambda W\|_{\dot H^1}^2 $
	and by parity,
	$
	\pshbbun{ \nabla W_1 }{ \Lambda_1 W_1 }=0$.
	Using Claim~\ref{WW}, we have, for $k\in\{2,\dots,K\}$,
	\[
	\left| \pshbbun{ \Lambda_k W_k}{ \Lambda_1 W_1 }\right| 
	+ \left| \pshbbun{ \nabla W_k}{ \Lambda_1 W_1 } \right|
	\lesssim t^{-3}.
	\]
	And then, using the expression of ${\rm Mod}_{\mathbf{V}}$ and~\eqref{e:n40}, we obtain
	\begin{equation*}
	\left|({\rm Mod}_{\mathbf{V}},\Lambda_{1}W_{1})_{\dot{H}_{\bell}^{1}}\right|\lesssim \frac{1}{t^{2}}\sum_{k=1}^{K}\left(\left|\frac{\dot{\lambda_{k}}}{\lambda_{k}}\right|+|\dot{\mathbf{y}_{k}}|\right)\lesssim \frac{1}{t^{2}}\sum_{k=1}^{K}\left(\left|\frac{\dot{\lambda_{k}}}{\lambda_{k}}-\frac{a_{k}}{\lambda_{k}^{\frac{1}{2}} t^{2}}\right|+|\dot{\mathbf{y}_{k}}|\right)+t^{-4}.
	\end{equation*}
	In conclusion of the previous estimates, the orthogonality condition $(\varepsilon,\Lambda_{1}W_{1})_{\dot{H}^{1}_{\bell_{1}}}=0$, gives the following
	\[
	\left|\frac{\dot \lambda_1}{\lambda_1} - \frac{a_1}{\lambda_2^{\frac 12}t^2}\right| \lesssim \left\|\vec \varepsilon\right\|_{\dot H^1\times L^2}+\left(\left\|\vec \varepsilon\right\|_{\dot H^1\times L^2}+t^{-2}\right)\sum_{k=1}^{K}\left(\left|\frac{\dot{\lambda_{k}}}{\lambda_{k}}-\frac{a_{k}}{\lambda_{k}^{\frac{1}{2}}t^{2}}\right|+|\dot{\mathbf{y}_{k}}|\right)+t^{-4}.
	\]
	Using the other orthogonality conditions, 
	we obtain similarly,
	\begin{align*}
	&\sum_{k=1}^{K} \left(\left|\frac{\dot\lambda_k}{\lambda_k} - \frac{a_k}{\lambda_k^{\frac 12}t^2}\right|+\left|\dot {\mathbf{y}}_k \right| \right)\\ 
	&\lesssim \left\|\vec \varepsilon\right\|_{\dot H^1\times L^2}+\left(\left\|\vec \varepsilon\right\|_{\dot H^1\times L^2}+t^{-2}\right)\sum_{k=1}^{K}\left(\left|\frac{\dot{\lambda_{k}}}{\lambda_{k}}-\frac{a_{k}}{\lambda_{k}^{\frac{1}{2}}t^{2}}\right|+|\dot{\mathbf{y}_{k}}|\right)+t^{-4}.
	\end{align*} 
	Therefore, for $\delta_0$ small enough and $T_0$ large enough, we find~\eqref{le:p}. 
	\smallbreak
	
	\textbf{Step 4.} Equations of the unstable directions. We prove the case of $z_{1}^{+}$. The other case is similar.
	Recall that $\vec Z_{{\boldsymbol{\ell}}_k}^{\pm} \in \mathcal S$ by their definition in \S\ref{sec.23}.
	By~\eqref{syst_e3}, we have
	\begin{align*}
	\frac d{dt} z_1^{+} & =\frac d{dt} \psl {\vec \varepsilon}{\vec {\tilde \theta}_1\vec Z_{{\boldsymbol{\ell}}_1}^+} = 
	\psl {\partial_t \vec \varepsilon}{\vec {\tilde \theta}_1 \vec Z_{{\boldsymbol{\ell}}_1}^{+}}+
	\psl {\vec \varepsilon}{\partial_t\left(\vec {\tilde \theta}_1 \vec Z_{{\boldsymbol{\ell}}_1}^{+}\right)} \\
	& = \psl {\vec {\mathcal L} \vec \varepsilon}{\vec {\tilde \theta}_1 \vec Z_{{\boldsymbol{\ell}}_1}^{+}}
	- \frac {{\boldsymbol{\ell}}_1}{\lambda_1} \cdot \psl{\vec \varepsilon}{\vec{\tilde \theta}_1 \nabla\vec Z_{\ell_1}^{+}} 
	-\frac{\dot\lambda_1}{\lambda_1} \psl{\vec \varepsilon}{\vec{\tilde \theta}_1 \vec\Lambda \vec Z_{{\boldsymbol{\ell}}_1}^{+}} - \frac {\dot {\mathbf{y}}_1}{\lambda_1} \cdot \psl{\vec \varepsilon}{\vec{\tilde \theta}_1 \nabla\vec Z_{{\boldsymbol{\ell}}_1}^{+}} 
	\\ &\quad + \psl{\vec {\rm Mod_{1}}}{\vec {\tilde \theta}_1 \vec Z_{{\boldsymbol{\ell}}_1}^{+}} +\psl{\vec {\rm Mod_{2}}}{\vec {\tilde \theta}_1 \vec Z_{{\boldsymbol{\ell}}_1}^{+}}+ \psl{\vec{\mathbf{R}}}{\vec {\tilde \theta}_1 \vec Z_{{\boldsymbol{\ell}}_1}^{+}} +\psl{\vec{\mathbf{R}}_{\rm NL}}{\vec {\tilde \theta}_1 \vec Z_{{\boldsymbol{\ell}}_1}^{+}}.
	\end{align*}
	First, by direct computations, using~(ii) of Lemma~\ref{surZZ},
	\begin{align*}
	\psl {\vec {\mathcal L} \vec \varepsilon}{\vec {\tilde \theta}_1 \vec Z_{{\boldsymbol{\ell}}_1}^{+}}
	- \frac {{\boldsymbol{\ell}}_1}{\lambda_1} \cdot \psl{\vec \varepsilon}{\vec{\tilde \theta}_1 \nabla\vec Z_{{\boldsymbol{\ell}}_1}^{+}} &= \frac 1{\lambda_1} \psl{\vec \varepsilon}{\vec{\tilde \theta}_1 \left( -H_{{\boldsymbol{\ell}}_1} J \vec Z_{{\boldsymbol{\ell}}_1}^+\right)} + \sum_{k\ge 2}\psl \varepsilon {f'(W_k) (\theta_1 Z_{{\boldsymbol{\ell}}_1}^+)}\\
	& =  \frac {\sqrt{\lambda_0}}{\lambda_1} (1-|{\boldsymbol{\ell}}_1|^2)^{\frac 12} z_1^+ 
	+ \sum_{k\ge 2}\psl \varepsilon {f'(W_k) (\theta_1 Z_{{\boldsymbol{\ell}}_1}^+)}.
	\end{align*}
	By the decay properties of $\vec Z_{\bell_1}^+$ and Claim~\ref{WW},
	\[
	\sum_{k\ge 2}\left| \psl \varepsilon {f'(W_k) (\theta_1 Z_{{\boldsymbol{\ell}}_1}^\pm)}\right| \lesssim t^{-4} {\| \varepsilon\|_{\dot H^1}} .
	\]
	Next, by~\eqref{le:p},
	\[\left| \frac{\dot\lambda_1}{\lambda_1}\right|\left|\psl{\vec \varepsilon}{\vec{\tilde \theta}_1 \vec \Lambda \vec Z_{\bell_1}^{+}} \right| +
	\left|\frac {\dot {\mathbf{y}}_1}{\lambda_1} \cdot \psl{\vec \varepsilon}{\vec{\tilde \theta}_1 \nabla\vec Z_{{\boldsymbol{\ell}}_1}^{+}} \right|
	\lesssim \left\|\vec\varepsilon\right\|_{\dot H^1\times L^2}^2 +t^{-2}\left\|\vec \varepsilon\right\|_{\dot H^1\times L^2}.
	\]
	Concerning the term with $\vec{\rm Mod}_{1}$.
	From (ii) of Lemma~\ref{surZZ}, we obtain
	\begin{equation*}
(\vec \theta_1 \vec Z_{{\boldsymbol{\ell}}_1}^{\Lambda},\vec {\tilde \theta}_1 \vec Z_{{\boldsymbol{\ell}}_1}^{+})= (\vec Z_{{\boldsymbol{\ell}}_1}^{\Lambda},\vec Z_{{\boldsymbol{\ell}}_1}^{+}) =0,\quad (\vec \theta_1 \vec Z_{{\boldsymbol{\ell}}_1}^{\nabla},\vec {\tilde \theta}_1 \vec Z_{{\boldsymbol{\ell}}_1}^{+})= (\vec Z_{{\boldsymbol{\ell}}_1}^{\nabla},\vec Z_{{\boldsymbol{\ell}}_1}^{+}) =0.
\end{equation*}
	Moreover, by Claim~\ref{WW}, we have
	\[
	\sum_{k=2}^{K}\left(\left|\psl {\vec\theta_k \vec Z_{\bell_k}^{\Lambda}}
	{\vec {\tilde \theta}_1 \vec Z_{\bell_1}^{+}}\right|
	+\left|\psl{\vec \theta_k \vec Z_{\bell_k}^{\nabla}}{\vec {\tilde \theta}_1 \vec Z_{\bell_1}^{+}}\right|\right) \lesssim t^{-3},
	\]
	and thus, by~\eqref{le:p},
	\[
	\sum_{k=2}^{K}\left(\left|\frac{\dot\lambda_k}{\lambda_k} - \frac{a_k}{\lambda_k^{\frac 12}t^2}\right| \left|\psl {\vec\theta_k \vec Z_{\bell_k}^{\Lambda}}
	{\vec {\tilde \theta}_1 \vec Z_{\bell_1}^{+}}\right|+
	\left| \frac {\dot {\mathbf{y}}_k}{\lambda_k} \cdot\psl{\vec \theta_k \vec Z_{\bell_k}^{\nabla}}{\vec {\tilde \theta}_1 \vec Z_{\bell_1}^{+}}\right|\right) \lesssim 
	t^{-3} \left\|\vec\varepsilon\right\|_{\dot H^1\times L^2}+t^{-7}.
	\]
	Concerning the term with $\vec{\rm Mod}_{2}$. From~\eqref{e:n40}, ~\eqref{le:p} and the definition of $\vec{\rm Mod}_{2}$, we obtain
	\begin{equation*}
	\left|\psl{\vec {\rm Mod_{2}}}{\vec {\tilde \theta}_1 \vec Z_{{\boldsymbol{\ell}}_1}^{+}}\right|\lesssim \frac{1}{t^{2}}\sum_{k=1}^{K}\left(\left|\frac{\dot{\lambda_{k}}}{\lambda_{k}}\right|+|\dot{\mathbf{y}_{k}}|\right)\lesssim t^{-2}\left\|\vec\varepsilon\right\|_{\dot H^1\times L^2}+t^{-4}
	\end{equation*}
	Finally, we claim
	\[
	\left| \psl{\vec {\mathbf{R}}}{\vec {\tilde \theta}_1 \vec Z_{\bell_1}^{+}}\right|+
	\left| \psl{\vec {{\mathbf R}}_{\rm NL}}{\vec {\tilde \theta}_1 \vec Z_{\bell_1}^{+}}\right|
	\lesssim
	t^{-4+\delta} + t^{-1} {\| \varepsilon\|_{\dot H^1}} + \| \varepsilon\|_{\dot H^1}^2.
	\]
	Indeed, from~\eqref{RWX}, we have
	$
	|(\vec {\mathbf{R}},\vec {\tilde \theta}_1 \vec Z_{\bell_1}^{+})|\lesssim
	t^{-4+\delta}$.
	Second, by~\eqref{R1bis} and the decay of~$Y$, we have
	\begin{equation*}
\left|\left(\vec{{\mathbf R}}_{1},\vec{\tilde \theta}_1 \vec Z_{\bell_1}^{+}\right)\right|\lesssim \|\mathbf{R}_{1}\|_{\frac{10}{7}}\lesssim t^{-1}\|\varepsilon\|_{\dot{H}^{1}} \quad \mathrm{and} \quad \left|\left(\vec{{\mathbf R}}_{2},\vec{\tilde \theta}_1 \vec Z_{\bell_1}^{+}\right)\right|\lesssim \|\mathbf{R}_{2}\|_{\frac{10}{7}}\lesssim \|\varepsilon\|_{\dot{H}^{1}}^{2} 
\end{equation*}

Gathering these estimates, we obtain~\eqref{le:z}. The proof of Lemma~\ref{le:4} is complete.
\end{proof}
\section{Proof of Theorem~\ref{th:1}}
To construct the $K$-soliton solution at $+\infty$, we follow the strategy of~\cite{MMwave1} using the refined approximate solution $\vec{\mathbf W}$ defined in the previous section. We argue by compactness and obtain the solution $u(t)$ as the limit of a sequence of approximate multi-solitons $u_n(t)$.
\begin{proposition}\label{pr:S1}
	There exist $T_0>0$
	and a solution $u(t)$ of~\eqref{wave} on $[T_0,+\infty)$ satisfying, for all $0<\delta<1$, for all $t\in [T_0,+\infty)$,
	\begin{equation}\label{pr:S11}
	\left\|\nabla u(t) - \nabla {\mathbf W}(t)\right\|_{L^2}
	+\left\| \partial_t u(t) - {\mathbf X}(t) \right\|_{L^2} \lesssim t^{-3+\delta}
	\end{equation}
	where $\lambda_k(t)$, $\mathbf{y}_k(t)$ are such that, for all $t\in [T_0,+\infty)$,
	\begin{equation}\label{blay}
	|\lambda_k(t)-\lambda_k^\infty|+|\mathbf{y}_k(t)-\mathbf{y}_k^\infty|\lesssim t^{-1}.
	\end{equation}
\end{proposition}
This section is devoted to the proof of Proposition~\ref{pr:S1}. Note that Proposition~\ref{pr:S1} implies Theorem~\ref{th:1}.

\smallbreak

Let $S_n\to +\infty$.
Let $\zeta_{k,n}^{\pm}\in {\mathbb{R}}$ small to be determined later. These free parameters correspond to two exponentially stable/unstable directions for each soliton - see statements of Proposition~\ref{pr:s5}, Claim~\ref{le:modu2} and Lemma~\ref{le:bs2}.
For any large $n$, we consider the solution $u_n$ of
\begin{equation}\label{defun}
\left\{
\begin{aligned} 
& \partial_t^2 u_n - \Delta u_n - |u_n|^{\frac 4{3}} u_n = 0 \\ 
& (u_n(S_n),\partial_t u_n(S_n))^\mathsf{T} = \sum_{k=1}^{K} \left( \vec W_{k}^\infty (S_n) + c_k \vec V_k^{\infty}(S_n)
+\sum_\pm \zeta_{k,n}^\pm ( \vec \theta_k^\infty \vec Z_{{\boldsymbol{\ell}}_k}^\pm)(S_n)\right).
\end{aligned}
\right.
\end{equation}
Note that since $(u_n(S_n),\partial_t u_n(S_n))\in \dot H^1\times L^2$, the solution $\vec u_n$ is well-defined in $\dot H^1\times L^2$ at least on a small interval of time around $S_n$.

Now, we state uniform estimates on $u_n$ backwards in time up to some uniform $T_0\gg1$.
\begin{proposition}\label{pr:s5} There exist $n_0>0$ and $T_0>0$ such that, for any $n\geq n_0$, there exist
	$(\zeta_{k,n}^{\pm})_{k\in\{1,\dots,K\}}\in {\mathbb{R}}^K\times {\mathbb{R}}^K$, with
	\[
	\sum_{k=1}^{K} |\zeta_{k,n}^{+}|^2+\sum_{k=1}^{K} |\zeta_{k,n}^{-}|^2 \lesssim {S_n^{-7}},
	\]
	and such that the solution $\vec u_n=(u_n,\partial_t u_n)^\mathsf{T}$ of~\eqref{defun} is well-defined in $\dot H^1\times L^2$ on the time interval $[T_0,S_n]$
	and satisfies, for all $t\in [T_0,S_{n}]$,
	\begin{equation}\label{eq:un} 
	\left\|\vec u_n(t) -\vec {\mathbf W}_{n}(t)\right\|_{\dot H^1\times L^2} \lesssim t^{-3+\delta}
	\end{equation}
and for all $\nu>0$, there exists  $M>0$ such that, for all $n\geq n_0$,
\begin{equation}\label{cpq}
\|\vec u_n(T_0)\|_{(\dot H^1\times L^2)(|x|>M)} <\nu,
\end{equation}
	where $
	\vec {\mathbf W}_{n}(t,x)=
	\vec {\mathbf W}\left(t,x;\{\lambda_{k,n}(t)\},\{\mathbf{y}_{k,n}(t)\}\right)$ is defined in \S 4,
	and
	\begin{equation}\label{eq:deux}
	|\lambda_{k,n}(t)-\lambda_k^\infty|+|\mathbf{y}_{k,n}(t)-\mathbf{y}_{k}^\infty|\lesssim t^{-1},\quad
	\left| \dot \lambda_{k,n} (t)\right|+\left|\dot {\mathbf{y}}_{k,n} (t)\right|\lesssim t^{-2}.
	\end{equation}
	Moreover, $\vec u_n\in \mathcal C([T_0,S_n],\dot H^2\times \dot H^1)$ and satisfies, for all $t\in [T_0,S_n]$,
	$\|\vec u_n(t)\|_{\dot H^2\times \dot H^1} \lesssim 1$.
\end{proposition}
\subsection{Proof of Proposition~\ref{pr:S1}, assuming Proposition~\ref{pr:s5}}
In view of the estimates obtained in Proposition~\ref{pr:s5} on $(\vec u_n(T_0))$ and~\eqref{cpq}, up to the extraction of a subsequence,  $(\vec u_n(T_0))$ converges strongly in $\dot H^1\times L^2$ to some $(u_0,u_1)^\mathsf{T}$ as $n\to +\infty$.
Consider the solution $u(t)$ of~\eqref{wave} associated to the initial data $(u_0,u_1)^\mathsf{T}$ at $t=T_0$.
Then, by the uniform bounds~\eqref{eq:un} and the continuous dependence of the solution of~\eqref{wave} with respect to its initial data in the energy space $\dot H^1\times L^2$ 
(see \emph{e.g.}~\cite{KM} and references therein), the solution $u$ is well-defined in the energy space on $[T_0,\infty)$.

Recall that we denote by $\lambda_{k,n}$ and $\mathbf y_{k,n}$ the parameters of the decomposition of $u_n$ on $[T_0,S_n]$.
By the uniform estimates in~\eqref{eq:deux}, using Ascoli's theorem and a diagonal argument, it follows that there exist continuous functions $\lambda_k$ and $\mathbf y_k$ such that up to the extraction of a subsequence, $\lambda_{k,n}\to \lambda_k$, $\mathbf y_{k,n}\to \mathbf y_k$ uniformly on compact sets of $[T_0,+\infty)$.
Moreover on $[T_0,+\infty)$,
\[
|\lambda_k(t)-\lambda_k^\infty|\lesssim t^{-1},\quad |\mathbf y_k(t)-\mathbf y_k^\infty|\lesssim t^{-1}.
\]
Passing to the limit in~\eqref{eq:un} for any $t\in [T_0,+\infty)$, 
we finish the proof of Proposition~\ref{pr:S1}.

The rest of this section is devoted to the proof of Proposition~\ref{pr:s5}.
\subsection{Bootstrap setting}
We denote by $\mathcal B_{{\mathbb{R}}^K}(r)$ (respectively, $\mathcal S_{{\mathbb{R}}^K}(r)$) the open ball (respectively, the sphere) of ${\mathbb{R}}^K$ of center $0$ and of radius $r>0$, for 
the norm $|(\xi_k)_k|= (\sum_{k=1}^{K} \xi_k^2 )^{1/2}$.

For $t=S_n$ and for $t<S_n$ as long as $u_{n}(t)$ is well-defined in $\dot{H}^1\times L^2$ and satisfies~\eqref{hyp:4}, we decompose $u_n(t)$ as in Lemma~\ref{le:4}. In particular, we denote by $(\varepsilon,\eta)$, $(\lambda_k)_k$, $(\mathbf{y}_k)_k$, $(z_k^{\pm})_k$ the parameters of the decomposition of $u_n$.

We start with a technical result similar to Lemma 3 in \cite{CMM}. This claim will allow us to adjust the initial values of
$(z_k^\pm(S_n))_k$  from the choice of $\zeta_{k,n}^\pm$ in \eqref{defun}.
\begin{claim}[Choosing the initial unstable modes] \label{le:modu2}
	There exist $n_0>0$ and $C>0$ such that, for all $n\geq n_0$, for any
	$(\xi_k)_{k\in\{1,\dots,K\}}\in \overline{\mathcal B}_{{\mathbb{R}}^K}(S_n^{-7/2})$, there exists
	a unique $(\zeta_{k,n}^{\pm})_{k\in\{1,\dots,K\}}\in \mathcal B_{{\mathbb{R}}^{2K}}(C S_n^{-7/2})$ such that
	the decomposition of $u_n(S_n)$ satisfies
	\begin{equation}\label{modu:2}
	z_k^-(S_n)=\xi_k,\quad
	z_k^+(S_n)=0,
	\end{equation}
	\begin{equation}\label{modu3}
	|\lambda_k(S_n)-\lambda_k^\infty|+
	|\mathbf{y}_k(S_n)-\mathbf{y}_k^\infty|
	+ \left\|\vec \varepsilon(S_n)\right\|_{\dot H^1\times L^2}+ \left\|\vec \varepsilon(S_n)\right\|_{\dot H^2\times \dot{H}^{1}}\lesssim S_n^{-7/2}.
	\end{equation}
\end{claim}
\begin{proof}[Sketch of the proof of Claim~\ref{le:modu2}]
	The proof of existence of $(\zeta_{k,n}^{\pm})_k$ in Claim~\ref{le:modu2} is similar to Lemma~3 in~\cite{CMM} and we omit it. Estimates in~\eqref{modu3} are consequences of~\eqref{bounds}.
\end{proof}
The proof of Proposition~\ref{pr:s5} is based on the following bootstrap estimates, for some $0<\delta\ll 1$ to be fixed later, $k\in\{1,\dots,K\} $ and $C_{0} $ to be chosen,
\begin{equation}\label{eq:BS}\left\{\begin{aligned}
& |\lambda_k(t)-\lambda_k^\infty|\leq C_0 t^{-1},\quad |\mathbf{y}_k(t)-\mathbf{y}_k^\infty|\leq C_0 t^{-1}, \\
& |z_k^\pm(t)|^2\leq t^{-7},\quad 
\left\|\vec \varepsilon(t)\right\|_{\dot H^1\times L^2}\leq C_{0}t^{-3+\delta}.\end{aligned}\right.
\end{equation}
Set
\begin{equation}\label{def:tstar}
T^*=T_n^*((\xi_{k})_k)=\inf\{t\in[T_0,S_n]\ ; \ \hbox{$u_n$ satisfies~\eqref{hyp:4} and (\ref{eq:BS}) holds on $[t,S_n]$} \} .
\end{equation}
In what follows, we will prove that there exists $T_0$ large enough and at least one choice of $(\xi_{k})_{k}\in \mathcal B_{{\mathbb{R}}^K}(S_n^{-7/2})$ so that $T^*=T_0$, which is enough to finish the proof of Proposition~\ref{pr:s5}.
For this, we derive general estimates for any $(\xi_k)_k\in \overline{\mathcal B}_{{\mathbb{R}}^K}(S_n^{-7/2})$ (see Lemma~\ref{le:bs1}) and use a topological argument (see Lemma~\ref{le:bs2}) to control the unstable directions, in order to strictly improve~\eqref{eq:BS} on $[T^*,S_n]$.

\smallbreak
As a consequence of the bootstrap estimates~\eqref{le:p} and~\eqref{eq:BS}, we have, for $k\in\{1,\dots,K\}$,
\[
\left|\frac {\dot \lambda_k }{\lambda_k} - \frac{a_k}{\lambda_k^{\frac 12}t^2}\right|+|\dot {\mathbf{y}}_k| \lesssim \left\|\vec \varepsilon(t)\right\|_{\dot H^1\times L^2} +t^{-4+\delta} \lesssim C_{0}t^{-3+\delta}.
\]
In particular, from the expression of ${\rm Mod}_{{\mathbf V}}$ and ${\rm Mod}_{{\mathbf Z}}$ in Lemma~\ref{le:43},~\eqref{e:n40},~\eqref{rk33},~\eqref{le:p} and~\eqref{eq:BS}, we have
\begin{equation}\label{lemvz}
\|{\rm Mod}_{{\mathbf V}}\|_{\dot{H}^{2}\cap\dot{H}^{1}}\le t^{-4}\quad \mathrm{and}\quad \|{\rm Mod}_{{\mathbf Z}}\|_{{H}^{1}}\le t^{-4+\delta}.
\end{equation}
From the expression of ${\rm Mod}_{{\mathbf W}}$ and ${\rm Mod}_{{\mathbf X}}$ in Lemma~\ref{le:43}, for all $\alpha \in {\mathbb{N}}^5$,\begin{equation}\label{Idem}
|\partial_x^\alpha{\rm Mod}_{{\mathbf W}}(t)|\lesssim t^{-3+\delta} \sum_{k=1}^{K} |W_k|^{1+\frac{|\alpha|}{3}},\quad
|\partial_x^\alpha{\rm Mod}_{{\mathbf X}}(t)|\lesssim t^{-3+\delta} \sum_{k=1}^{K} |W_k|^{1+\frac{1+|\alpha|}{3}}.
\end{equation}
\subsection{Energy functional}
Recall that 
	$f(u)=|u|^{\frac 43} u$ and $F(u)=\frac 3{10} |u|^{\frac {10}3}$.
Let
\begin{equation}\label{hyp:sec3}
\forall k\in\{1,\ldots,K\},\quad 
\boldsymbol\ell_k = \sum_{i=1}^{5}\ell_{k,i} \mathbf{e}_i \quad \hbox{where}\quad |\boldsymbol\ell_k| < 1,
\end{equation}
and denote (see~\eqref{sc})
\begin{equation}\label{spe boud}
\overline{\ell}=\big(\sum_{i=1}^{5}(\max_{k\in\{1,\cdots,K\}}\ell^{2}_{k,i})\big)^{\frac{1}{2}}<\frac{3}{5}.
\end{equation}
We fix $\delta<\frac{1}{100}$ small enough, such that
\begin{equation}\label{bd}
\begin{aligned}
& \overline{\ell}\le \left(\frac{(2-4\delta)(8-4\delta)}{(6-4\delta)^{2}}\right)^{\frac{1}{2}},\quad \overline{\ell}\le\left(\frac{(1-4\delta)(9-4\delta)}{(5-4\delta)^{2}}\right)^{\frac{1}{2}},\\
& \overline{\ell}\le \frac{(5-4\delta)}{(6-4\delta)},\quad \overline{\ell}\le \left(\frac{4-4\delta}{6-4\delta}\right)^{\frac{1}{2}},\quad \overline{\ell}\le\left(\frac{(3-4\delta)(7-4\delta)}{(6-4\delta)^{2}}\right)^{\frac{1}{2}}.
\end{aligned}
\end{equation}
Moreover, we set,
\begin{equation*}
\forall i\in \{1,\ldots,5\}\quad -1<\overline\ell_{1,i}<\cdots<\overline\ell_{K_{i},i}<1\quad \mathrm{such\ that}\ \{\ell_{1,i},\cdots,\ell_{K,i}\}=\{\overline{\ell}_{1,i}\cdots,\overline{\ell}_{K_{i},i}\}.
\end{equation*}
We denote $I=\{i|i\in \{1,\cdots,5\}\ \mathrm{and}\ K_{i}\ge 2\}$. For
\begin{equation*}
0<\sigma<\frac 1{10} \min_{i\in I}(\overline\ell_{k+1,i}-\overline\ell_{k,i})
\end{equation*}
small enough to be fixed, we set
\begin{align*}
\hbox{for $k=1,\dots, K_{i}-1$},\quad 		&\overline\ell_{k,i}^{+}=\overline\ell_{k,i}+\sigma(\overline\ell_{k+1,i}-\overline\ell_{k,i}),\\
\hbox{for $k=2,\dots, K_{i}$},\quad  	&\overline\ell_{k,i}^{-}=\overline\ell_{k,i}-\sigma(\overline\ell_{k,i}-\overline\ell_{k-1,i}),
\end{align*}
and for $t>0$, we denote, 
$$
\Omega_{1,i}(t) =  ( (\overline\ell_{1,i}^+ t,\overline\ell_{2,i}^- t)\cup\ldots \cup (\overline\ell_{K_{i}-1,i}^+ t,\overline\ell_{K_{i},i}^- t)) ,\quad
\Omega_{0,i}(t) = \mathbb{R}\setminus \Omega_{1,i}(t)\quad \mathrm{for}\ i\in I.
$$
and
$$ \Omega_{1,i}(t)=\emptyset ,\quad \Omega_{0,i}(t) = \mathbb{R}\quad \mathrm{for}\ i\notin I.$$

When $i\in I$, we consider the continuous function $\chi_i(t,x)=\chi_{i}(t,x_i)$ defined as follows, for all $t>0$,
\begin{equation}\label{defchiK}
\left\{\begin{aligned}
&\hbox{$\chi_{i}(t,x) =  \overline\ell_{1,i}$ for $x_i\in (-\infty,  \overline\ell_{1,i}^+ t]$},\\
& \hbox{$\chi_{i}(t,x) = \overline\ell_{k,i} $ for $x_i\in [\overline\ell_{k,i}^- t, \overline\ell_{k,i}^+ t]$, for $k\in \{2,\ldots,K_{i}-1\}$,}
\\ & \hbox{$\chi_{i}(t,x)= \overline\ell_{K_{i},i}$ for $x_i\in [\overline\ell_{K_{i},i}^- t,+\infty)$},\\
& \chi_{i}(t,x) = \frac{x_i}{(1-2\sigma)t}  - \frac {\sigma}{1-2 \sigma} (\overline\ell_{k+1,i}+\overline\ell_{k,i}) 
\hbox{ for $x_i \in [\overline\ell_{k,i}^+ t,\overline\ell_{k+1,i}^-t ]$, $k\in \{1,\ldots,K_{i}-1\}$}.
\end{aligned}
\right.
\end{equation}
In particular,
\begin{equation}\label{derchi}\left\{\begin{aligned}
& \partial_t \chi_i(t,x) =0,\quad  \nabla \chi_i(t,x)=0, \quad \mathrm{on}\ \mathbb{R}^{i-1}\times\hbox{$\Omega_{0,i}(t)$}\times \mathbb{R}^{5-i},\\
& \partial_{x_i} \chi_i(t,x)= \frac{1}{(1-2\sigma)t}  \quad \mathrm{for}\ x\in\mathbb{R}^{i-1}\times\hbox{$\Omega_{1,i}(t)$}\times \mathbb{R}^{5-i},\\
& \partial_{t} \chi_i(t,x)= -\frac 1t \frac{x_i}{(1-2\sigma)t} \quad \mathrm{for}\ x\in\mathbb{R}^{i-1}\times\hbox{$\Omega_{1,i}(t)$}\times \mathbb{R}^{5-i}.
\end{aligned} \right.\end{equation}
When $i\notin I$, we consider the continuous function $\chi_i(t,x)=\chi_{i}(t,x_i)=\ell_{1,i}$, for all $t>0$.
\\ We denote
\begin{equation*}
\Omega=\bigcup_{\sum_{i} s_{i}\ne 0}\Omega_{s_{1},1}\times\cdots\times\Omega_{s_{5},5}\quad \mathrm{and}\quad \vec{\chi}=(\chi_1,\cdots,\chi_5).
\end{equation*}
The choice of $\vec{\chi}$ in this paper is different from that in~\cite{MMwave1,MMwave2} to take into account non-colinear speeds.

We define (see~\cite{MMT,MMwave1,MMwave2,CMkg} for similar energy functional)
\[\mathcal H =\int \left\{|\nabla\varepsilon|^2+|\eta|^2- 2 (F ({\mathbf W} +\varepsilon )-F ({\mathbf W} )-f ({\mathbf W} )\varepsilon ) \right\}
+ 2 \int (\vec{\chi}\cdot\nabla \varepsilon )\eta.
\] 
\begin{lemma}\label{mainprop}
	There exists $\mu>0$ such that, for $t\in[T^{*},S_{n}] $, the following hold.
	\begin{enumerate}[label=\emph{(\roman*)}]
		\item\emph{Bound.}
		\begin{equation}\label{boun}
		|\mathcal H(t)| \leq \frac {\|\vec \varepsilon(t)\|_{\dot H^1\times L^2}^2}{\mu}.
		\end{equation}
		\item\emph{Coercivity.}
		\begin{equation}\label{coer}
		\mathcal H(t) \geq \mu\|\vec \varepsilon(t)\|_{\dot H^1\times L^2}^2 - \frac {t^{-7}}{\mu} .
		\end{equation}
		\item\emph{Time variation.} For all $0<\delta<\frac{1}{100}$ small enough to satisfy~\eqref{bd},
		\begin{equation}\label{time}
		- \frac d{dt} \left( t^{6-3\delta} \mathcal H\right) \le \frac{1}{\mu}C_{0}t^{-1-\delta}.
		\end{equation}
	\end{enumerate}
\end{lemma}
\begin{proof}[Proof of Lemma~\ref{mainprop}]
	\textbf{Proof of~\eqref{boun}.}
	Since
	\[
	| F({\mathbf W}+\varepsilon) - F({\mathbf W}) - f({\mathbf W}) \varepsilon|\lesssim |\varepsilon|^{\frac {10}3} + |\varepsilon|^2|{\mathbf W}|^{\frac 43},
	\]
	estimate~\eqref{boun} on $\mathcal H$ follows from~\eqref{sobolev},~\eqref{holder1} and $\|\vec \varepsilon\|_{\dot H^1\times L^2}+\|{\mathbf W}\|_{\dot H^1}\lesssim 1$.
	
	\smallbreak
	\textbf{Proof of~\eqref{coer}.}
	Set
	\[ 
	{\mathcal N_{\Omega}}(t) = \int_{\Omega} \left( |\nabla \varepsilon(t)|^2 + \eta^2(t) + 2 (\vec{\chi}(t)\cdot\nabla \varepsilon(t) ) \eta(t)\right)
	,\quad 
	{\mathcal N_{\Omega^C}}(t) = \int_{{\Omega^C}} \left( |\nabla \varepsilon(t)|^2 + \eta^2(t) \right).
	\]
	Note that, since $|\vec\chi|<\overline \ell$,
	\begin{equation}\label{nint} 
	{\mathcal N_{\Omega}}\ge \int_{\Omega} \left( |\nabla \varepsilon|^2 + \eta^2\right)-2\overline{\ell}\int_{\Omega} |\nabla \varepsilon||\eta|\ge(1-\overline \ell) \int_{\Omega}\left( |\nabla \varepsilon|^2 + \eta^2\right).
	\end{equation} 
	To obtain \eqref{coer}, we claim the following estimates, for some small $\alpha>0$
\begin{equation}\label{lF} 
\mathcal H(t) \geq 
{\mathcal N_{\Omega}}(t) + \mu {\mathcal N_{\Omega^C}}(t) -\frac {t^{-7}}\mu - \frac{t^{-4\alpha}}\mu \left\|\vec \varepsilon\right\|_{\dot H^1\times L^2}^2 - \frac{t^{-1}}\mu \left\|\vec \varepsilon\right\|_{\dot H^1\times L^2}^2- \frac{1}\mu \left\|\vec \varepsilon\right\|_{\dot H^1\times L^2}^3.
\end{equation}
	To prove~\eqref{lF}, we decompose 
	$
	\mathcal H = {\mathcal H_1}+ {\mathcal H_2} + {\mathcal H_3},
	$
	where
	\begin{align*}
	{\mathcal H_1} &
	= \int |\nabla \varepsilon|^2 - \sum_{k=1}^{K} \int f'(W_k) \varepsilon^2 + \int \eta^2 + 2\int (\vec{\chi}\cdot\nabla \varepsilon) \eta,
	\\{\mathcal H_2} &
	= - 2 \int \left( F\left({\mathbf W} +\varepsilon\right)-F\left({\mathbf W}\right) -f\left({\mathbf W}\right)\varepsilon - \frac 12 f' \left({\mathbf W}\right) \varepsilon^2 \right),
	\\{\mathcal H_3} &
	= \int \left( \sum_{k=1}^{K} f'(W_k) - f' \left({\mathbf W}\right) \right) \varepsilon^2.
	\end{align*}
	We claim the following estimates
	\begin{align}
	& {\mathcal H_1} \geq {\mathcal N_{\Omega}} + \mu {\mathcal N_{\Omega^C}} - \frac {t^{-7}}{\mu} - \frac{t^{-4\alpha}}\mu \left\|\vec \varepsilon\right\|_{\dot H^1\times L^2}^2, \label{bisff2}\\
	& |{\mathcal H_2}|+|{\mathcal H_3}|\lesssim \left\|\vec\varepsilon\right\|_{\dot H^1\times L^2}^3+t^{-1} \left\|\vec\varepsilon\right\|_{\dot H^1\times L^2}^2 .\label{bisff3}
	\end{align}
	which imply~\eqref{lF} for $T_0$ large enough.
	\smallbreak
	
	\emph{Proof of~\eqref{bisff2}.} For $\varphi_\alpha$ defined in~\eqref{phia}, set
	\[
	\varphi_k(t,x) = \varphi_\alpha \left(\frac {x-{\boldsymbol{\ell}}_{k} t - \mathbf{y}_{k}(t)}{\lambda_k(t)}\right).
	\]
	We decompose $\mathcal H_1$ as follows
\begin{align*}
{\mathcal H_1} & = {\mathcal N_{\Omega}} + \sum_{k=1}^{K} \left(\int |\nabla \varepsilon|^2\varphi_k^2 - \int f'(W_k)\varepsilon^2 + \int \eta^2 \varphi_k^2+ 2 \int (\vec{\chi}\cdot \nabla \varepsilon) \eta\varphi_k^2 \right)\\
&\quad +\int_{\Omega^C} \left(|\nabla \varepsilon|^2+\eta^2+2 (\vec{\chi}\cdot\nabla \varepsilon)\eta\right)\left(1-\sum_{k=1}^{K} \varphi_k^2\right)\\
&\quad- \int_{\Omega} \left(|\nabla \varepsilon|^2+\eta^2+2 (\vec{\chi}\cdot\nabla \varepsilon)\eta\right)\left(\sum_{k=1}^{K} \varphi_k^2\right)\\	&\quad +2 \sum_{k=1}^{K} \int \left((\vec{\chi}- \boldsymbol{\ell}_k)\cdot\nabla\varepsilon\right) \eta \varphi_k^2
={\mathcal N_{\Omega}}+ {\mathcal H_{1,1}}+{\mathcal H_{1,2}} +{\mathcal H_{1,3}}+{\mathcal H_{1,4}}.
\end{align*}
	By Lemma~\ref{surZZ} (iii), the orthogonality conditions on $\vec \varepsilon$ and a change of variable, we have
\begin{align*}
{\mathcal H_{1,1}} & \geq \mu_{0} \int \left( |\nabla \varepsilon|^2+\eta^2\right) \left(\sum_{k=1}^{K}\varphi_k^2\right)
-\frac 1{\mu_{0}} \sum_{k=1}^{K} \left( (z_k^-)^2 + (z_k^+)^2\right).
\end{align*}
Thus, using~\eqref{eq:BS},
\begin{align*}
{\mathcal H_{1,1}} & \geq \mu_{0} \int \left( |\nabla \varepsilon|^2+\eta^2\right) \left(\sum_{k=1}^{K}\varphi_k^2\right)
-\frac 1{\mu_{0}} t^{-7}
\geq \mu_{0} \int_{\Omega^C} \left( |\nabla \varepsilon|^2+\eta^2\right) \left(\sum_{k=1}^{K}\varphi_k^2\right)
-\frac 1{\mu_{0}}t^{-7}.
\end{align*}
Next, note that if $x$ is such that
$\varphi_k(t,x) >\frac 1K$, then $\varphi_{k'}^2(x) \lesssim t^{-4 \alpha}$ for $k'\neq k$.
Thus, there exists $\mu_{1}$ the estimate $1-\sum_{k=1}^{K} \varphi_k^2 \ge - \frac{t^{-4 \alpha}}{\mu_{1}}$ holds on ${\mathbb{R}}$.
By direct computations (with the notation $v_+ = \max (0,v)$,
\begin{align*}
{\mathcal H_{1,2}} & = \overline{\ell}\int_{\Omega^C} \left|\frac{\vec{\chi}\cdot\nabla\varepsilon}{\overline{\ell}}+\eta\right|^{2}\left(1-\sum_{k=1}^{K} \varphi_k^2\right)+(1-\overline{\ell})\int_{\Omega^C}\eta^{2}\left(1-\sum_{k=1}^{K} \varphi_k^2\right)\\
&+ \int_{\Omega^C} \left(|\nabla \varepsilon|^2-\frac{|\vec{\chi}\cdot\varepsilon|^{2}}{\overline{\ell}}\right)\left(1-\sum_{k=1}^{K} \varphi_k^2\right)\\
&\geq (1-\overline \ell) \int_{\Omega^C} \left(| \nabla \varepsilon|^2 + \eta^2\right)\left(1-\sum_{k=1}^{K} \varphi_k^2\right)_+ -\frac{\left\|\vec \varepsilon\right\|_{\dot H^1\times L^2}^2 }{\mu_{1}}t^{-4\alpha}.
\end{align*}
Last, by the definition of $\vec{\chi}$, the decay property of $\varphi_{\alpha}$ and $\eqref{eq:BS}$, we have
\begin{equation*}
\|\varphi_{k}^{2}\|_{L^{\infty}(\Omega)}\lesssim t^{-4\alpha}\quad\mathrm{and}\quad \|(\vec{\chi}-\boldsymbol{\ell}_{k})\varphi_{k}^{2}\|_{L^{\infty}}\lesssim t^{-4\alpha}.
\end{equation*}
Thus, $
|{\mathcal H_{1,3}}| \lesssim t^{-4\alpha} \left\|\vec \varepsilon\right\|_{\dot H^1\times L^2}^2 
$ and $
|{\mathcal H_{1,4}}| \lesssim t^{-4\alpha} \left\|\vec \varepsilon\right\|_{\dot H^1\times L^2}^2 
$.
Therefore, for some $\mu >0$, and $T_0$ large enough, we have
$$
{\mathcal H_{1,1}}+{\mathcal H_{1,2}}+{\mathcal H_{1,3}}+{\mathcal H_{1,4}}
\geq \mu \Nsol  - \frac 1\mu t^{-7} - \frac 1\mu t^{-4\alpha} \left\|\vec \varepsilon\right\|_{\dot H^1\times L^2}^2  .
$$
\emph{Proof of~\eqref{bisff3}.}
Using~\eqref{sobolev},~\eqref{holder1},~\eqref{e:WX} and~\eqref{eq:BS}, we have
\[
|{\mathcal H_{2}}| \lesssim\int|\varepsilon|^{\frac{10}3}+|\varepsilon|^3|{\mathbf W}|^{\frac13} \lesssim \left\|\vec\varepsilon\right\|_{\dot H^1\times L^2}^{3}.
\]
Last, we observe that by~\eqref{R1bis},
$
|{\mathcal H_{3}}| \lesssim \|{\mathbf{R}}_{1}\|_{L^{\frac {10}7}} \|\varepsilon\|_{L^{\frac{10}3}}^2 \lesssim t^{-1} \left\|\vec\varepsilon\right\|_{\dot H^1\times L^2}^2.
$ 
\smallbreak
\textbf{Proof of~\eqref{time}.} 
We decompose
\begin{align*}
\frac d{dt} \mathcal H =& \int \partial_t \left\{|\nabla\varepsilon|^2+|\eta|^2- 2 (F ({\mathbf W} +\varepsilon )-F ({\mathbf W} )-f ({\mathbf W} )\varepsilon ) \right\}\\
&+ 2 \int \vec\chi \cdot\partial_t\left( (\nabla \varepsilon) \eta \right) 
+ 2 \int (\partial_t \vec\chi) \cdot(\nabla \varepsilon) \eta = {\bf g_1} + {\bf g_2}.
\end{align*}
We claim the following estimates
\begin{align*} 
{\bf g_1} & = 2 \int \varepsilon \left( -\Delta {\rm Mod}_{{\mathbf W}} - f'({\mathbf W}) {\rm Mod}_{{\mathbf W}} \right) +2 \int \eta{\rm Mod}_{{\mathbf X}} \nonumber\\
& \quad + 2 \int \left( \sum_{k=1}^{K} \boldsymbol\ell_k \cdot \nabla  W_k\right) \left( f({\mathbf W} + \varepsilon) - f({\mathbf W})-f'({\mathbf W}) \varepsilon\right)
+ O\left(C_{0}t^{-7+2\delta}\right),
\end{align*}
\begin{align*}
{\bf g_2} & = 
\int(-\eta^{2}+|\nabla \varepsilon|^{2})\mathrm{div}\vec \chi -2\sum_{j,k=1}^{5}(\partial_{x_k}\varepsilon)(\partial_{x_j}\varepsilon)\partial_{x_j}\chi_{k}+2\sum_{k=1}^{5}(\partial_{x_k}\varepsilon)\eta\partial_{t}\chi_{k}\\
& \quad - 2\int \vec\chi\cdot\left( \sum_{k=1}^{K} \nabla W_k\right)\left( f({\mathbf W} + \varepsilon) - f({\mathbf W})-f'({\mathbf W}) \varepsilon\right)\nonumber\\
& \quad +2 \int (\vec\chi \cdot \nabla {\rm Mod}_{{\mathbf W}}) \eta - 2 \int \varepsilon (\vec\chi \cdot \nabla{\rm Mod}_{{\mathbf X}})
+ O\left(C_{0}t^{-7+2\delta}\right).
\end{align*}

\emph{Estimate on $\bf g_1$.}
From direct differentiation and integration by parts, we have
\begin{align*}
{\bf g_1} & = 2 \int (\partial_t \varepsilon) \left( -\Delta \varepsilon - \left( f({\mathbf W} + \varepsilon) - f({\mathbf W}) \right)\right) +2 \int (\partial_t \eta) \eta\\
& \quad -2 \int (\partial_t {\mathbf W}) \left( f({\mathbf W} + \varepsilon) - f({\mathbf W})-f'({\mathbf W}) \varepsilon \right) .
\end{align*}
Using~\eqref{syst_WX} and~\eqref{syst_e}
\begin{align*}
{\bf g_1} & = 2 \int \left(-\Delta \varepsilon -f'({\mathbf W})\varepsilon\right){\rm Mod}_{\mathbf W} + 2 \int \eta{\rm Mod}_{\mathbf X} \\
&\quad +2\int \left(-\Delta \varepsilon -f'({\mathbf W})\varepsilon\right){\rm Mod}_{\mathbf{V}} + 2\int \eta {{\mathbf R}}_{\mathbf X}+2\int \eta {\rm Mod}_{\mathbf Z} \\
& \quad -2 \int {\mathbf X} \left( f({\mathbf W} + \varepsilon) - f({\mathbf W})-f'({\mathbf W}) \varepsilon\right) 
={\bf g_{1,1}} + {\bf g_{1,2}}+ {\bf g_{1,3}}.\end{align*}
We integrate by parts terms in $\bf g_{1,1}$.
Next, by~\eqref{sobolev},~\eqref{holder1},~\eqref{RWX},~\eqref{e:WX},~\eqref{eq:BS} and~\eqref{lemvz}, we obtain
\begin{align*}
|{\bf g_{1,2}}|&\lesssim \|\varepsilon\|_{\dot H^1}\|{\rm Mod}_{\mathbf{V}}\|_{\dot H^1}+\|\eta\|_{L^2}\|{{\rm Mod}}_{\mathbf Z}\|_{L^2}\\
&+\|\eta\|_{L^2}\|{{\mathbf R}}_{\mathbf X}\|_{L^2}+\|\varepsilon\|_{L^{\frac{10}3}} \|{\mathbf W}\|_{L^{\frac{10}3}}^{\frac 43} \|{\rm Mod}_{\mathbf{V}}\|_{L^{\frac {10}3}}\lesssim C_{0}t^{-7+2\delta}.
\end{align*}
Recall from~\eqref{def:Wapp} that ${\mathbf X} = \sum_{k=1}^{K} \left(-\boldsymbol\ell_k \cdot\nabla W_k +c_k z_k\right)$. Moreover,
from~\eqref{e:n40} and the definition of $z_k$ in~\eqref{def.wk}, it follows that
$
\|z_k\|_{L^{\frac {10}3}}\lesssim \|\nabla z_k\|_{L^2}\lesssim t^{-2}.
$
Thus,
\begin{multline*}
\left|{\bf g_{1,3}} - 2 \int \left( \sum_{k=1}^{K} \boldsymbol\ell_k\cdot \nabla W_k\right) \left( f({\mathbf W} + \varepsilon) - f({\mathbf W})-f'({\mathbf W}) \varepsilon\right)\right|\\
\lesssim 
\int \left(\sum_{k=1}^{K} |z_k|\right) \left( |\varepsilon|^{\frac 73} +\varepsilon^2 |{\mathbf W}|^{\frac 13}\right)
\lesssim\sum_{k=1}^{K}\|z_{k}\|_{L^{\frac{10}{3}}}(\|\varepsilon\|_{L^{\frac{10}{3}}}^{\frac{7}{3}}+\|\varepsilon\|_{L^{\frac{10}{3}}}^{2})\lesssim C_{0}t^{-7+2\delta}.
\end{multline*}

\smallbreak
\emph{Estimate on $\bf g_2$.}
\begin{align*}
{\bf g_2} & = 2 \int (\vec \chi \cdot \nabla \partial_t \varepsilon) \eta+ 2 \int (\vec\chi \cdot\nabla \varepsilon) \partial_t \eta ++ 2 \int (\partial_t \vec\chi) \cdot(\nabla \varepsilon) \eta\\
& = 2 \int (\vec \chi \cdot \nabla \eta) \eta + 2 \int (\vec\chi\cdot \nabla \varepsilon) \left[ \Delta \varepsilon + \left( f({\mathbf W} + \varepsilon) - f({\mathbf W}) \right) \right]+ 2 \int (\partial_t \vec\chi) \cdot(\nabla \varepsilon) \eta\\
& \quad+2 \int (\vec \chi \cdot \nabla {\rm Mod}_{{\mathbf W}}) \eta + 2 \int (\vec\chi \cdot\nabla \varepsilon) {\rm Mod}_{{\mathbf X}}
+2 \int (\vec\chi \cdot \nabla {\rm Mod}_{\mathbf{V}}) \eta\\
&\quad +2\int (\vec\chi \cdot \nabla \varepsilon) {{\rm Mod} }_{{\mathbf Z}}+ 2 \int (\vec\chi \cdot \nabla \varepsilon) {{\mathbf R}}_{{\mathbf X}}.
\end{align*}
Note that by integration by parts 
\begin{align*}
&2 \int (\vec \chi \cdot \nabla \eta) \eta + 2 \int (\vec\chi\cdot \nabla \varepsilon)\Delta \varepsilon+2 \int (\partial_t \vec\chi) \cdot(\nabla \varepsilon) \eta\\
&=(-\eta^{2}+|\nabla \varepsilon|^{2})\mathrm{div}\vec \chi -2\sum_{j,k=1}^{5}(\partial_{x_k}\varepsilon)(\partial_{x_j}\varepsilon)\partial_{x_j}\chi_{k}+2\sum_{k=1}^{5}(\partial_{x_k}\varepsilon)\eta\partial_{t}\chi_{k}.
\end{align*}
Next, we observe
\begin{multline*}
\int (\vec\chi \cdot \nabla \varepsilon) \left( f({\mathbf W} + \varepsilon) - f({\mathbf W}) \varepsilon\right) 
\\= \int \vec\chi \cdot\nabla ( F({\mathbf W} + \varepsilon) - F({\mathbf W})-f({\mathbf W}) \varepsilon )
- \int \vec \chi \cdot(\nabla {\mathbf W}) \left( f({\mathbf W} + \varepsilon) - f({\mathbf W})-f'({\mathbf W}) \varepsilon\right).
\end{multline*}
Integrating by parts and using~\eqref{derchi},
\begin{align*}
| -\int \vec\chi \cdot\nabla ( F({\mathbf W} + \varepsilon) - F({\mathbf W})-f({\mathbf W}) \varepsilon )|&=|\int_{\Omega}\left(\mathrm{div}\vec{\chi}\right)\left( F({\mathbf W} + \varepsilon) - F({\mathbf W})-f({\mathbf W}) \varepsilon \right)|\\
&\lesssim  \frac 1{(1-2\sigma) t} \int_{\Omega} |\left( F({\mathbf W} + \varepsilon) - F({\mathbf W})-f({\mathbf W}) \varepsilon\right)|.
\end{align*}
Thus, by~\eqref{eq:BS} and 
\[
\|{\mathbf W}\|_{L^{\frac{10}3}(\Omega)}\lesssim \sum_{k=1}^{K} \left(\|W_k\|_{L^{\frac{10}3}(\Omega)}+\|v_k\|_{\dot H^1}\right)
\lesssim t^{-\frac 32},
\]
we obtain
\begin{align*}
\left| \int \vec\chi \cdot\nabla ( F({\mathbf W} + \varepsilon) - F({\mathbf W})-f({\mathbf W}) \varepsilon )\right|
&\lesssim t^{-1} \int_\Omega \left(|\varepsilon|^{\frac {10}3} + {\mathbf W}^{\frac 43} |\varepsilon|^2 \right)\\
&\lesssim t^{-1}(\|\varepsilon\|^{\frac{10}{3}}_{L^{\frac{10}{3}}}+\|\varepsilon\|_{L^{\frac{10}{3}}}^{2}\|\mathbf{W}\|^{\frac{4}{3}}_{L^{\frac{10}{3}}(\Omega)})\\
&\lesssim t^{-1}\big((C_{0})^{\frac{10}{3}}t^{-10+\frac{10}{3}\delta}+(C_{0}^{2}t^{-8+2\delta}\big)\lesssim C_{0}t^{-7+2\delta}.
\end{align*}
Moreover, again by~\eqref{e:n40} and~\eqref{eq:BS}
\begin{align*}
&\left| \int \vec\chi \cdot\left(\nabla {\mathbf W} -\nabla \sum_{k=1}^{K} W_k\right) \left( f({\mathbf W} + \varepsilon) - f({\mathbf W})-f'({\mathbf W}) \varepsilon\right)\right|\\
&= \left| \int \vec\chi\cdot \nabla \left(\sum_{k=1}^{K} c_k v_k\right) \left( f({\mathbf W} + \varepsilon) - f({\mathbf W})-f'({\mathbf W}) \varepsilon\right)\right|\\
&\lesssim \sum_{k=1}^{K} \int |\nabla v_k| \left( |\varepsilon|^{2} {\mathbf W}^{\frac 13} + |\varepsilon|^{\frac 73}\right)\lesssim \left( \sum_{k=1}^{K} \|\nabla v_k\|_{L^\frac{10}3} \right) \|\varepsilon\|_{L^{\frac {10}3}}^2\lesssim (C_{0})^{2}t^{-8+2\delta}\lesssim C_{0}t^{-7+2\delta}.
\end{align*}
Next, integrating by parts, 
\[
2 \int (\vec\chi\cdot\nabla \varepsilon) {\rm Mod}_{{\mathbf X}} 
= -2 \int (\vec\chi\cdot\nabla{\rm Mod}_{{\mathbf X}}) \varepsilon + O\left(C_{0}t^{-7+2\delta}\right),
\]
since by~\eqref{eq:BS},~\eqref{Idem} and~\eqref{derchi}
\begin{align*}
\left|\int (\mathrm{div}\vec\chi)\varepsilon {\rm Mod}_{{\mathbf X}}\right|
\lesssim t^{-4+\delta} \int_{\Omega} |\varepsilon| \left( \sum_{k=1}^{K} |W_k|^{\frac 43}\right)&
\lesssim t^{-4+\delta}\|\varepsilon\|_{L^{\frac{10}{3}}}\sum_{k=1}^{K}\|W_{k}\|_{L^{\frac{40}{21}}(\Omega)}^{\frac{4}{3}}\\
&\lesssim C_{0}t^{-7+2\delta}.
\end{align*}
We finish the estimate of~$\bf g_2$ by observing that~\eqref{RWX},~\eqref{eq:BS} and ~\eqref{lemvz}  yield
\begin{align*}
&\left| \int (\vec\chi \cdot \nabla {\rm Mod}_{\mathbf{V}}) \eta \right|+\left|\int (\vec\chi \cdot\nabla \varepsilon){\rm Mod}_{\mathbf{Z}}\right|+\left|\int (\vec\chi \cdot\nabla \varepsilon) {{\mathbf R}}_{{\mathbf X}}\right|\\
&\lesssim \|\eta\|_{L^2} \|\nabla {\rm Mod}_{\mathbf{V}}\|_{L^2} 
+\|\nabla \varepsilon\|_{L^2}  \|{\rm{Mod}}_{{\mathbf Z}}\|_{L^2}
+ \|\nabla \varepsilon\|_{L^2} \|{{\mathbf R}}_{{\mathbf X}}\|_{L^2}
\lesssim C_{0}t^{-7+2\delta}.
\end{align*}

\smallbreak
Gathering the estimates on $\bf g_1$ and $\bf g_2$, we rewrite
\[\frac d{dt} \mathcal H = {\bf h_1}+{\bf h_2}+{\bf h_3}+{\bf h_4}+ O\left(C_{0}t^{-7+2\delta}\right),\]
where
\begin{align*}
{\bf h_1} &= \int(-\eta^{2}+|\nabla \varepsilon|^{2})\mathrm{div}\vec \chi -2\sum_{j,k=1}^{5}(\partial_{x_k}\varepsilon)(\partial_{x_j}\varepsilon)\partial_{x_j}\chi_{k}+2\sum_{k=1}^{5}(\partial_{x_k}\varepsilon)\eta\partial_{t}\chi_{k}, \\ 
{\bf h_2}& = 2 \int \left(\sum_k \left(\boldsymbol\ell_k - \vec\chi\right) \cdot\nabla W_k\right) \left( f({\mathbf W} + \varepsilon) - f({\mathbf W})-f'({\mathbf W}) \varepsilon\right),\\
{\bf h_3} &= 2 \int \eta\left({\rm Mod}_{{\mathbf X}} + \vec\chi\cdot \nabla {\rm Mod}_{{\mathbf W}}\right),\\
{\bf h_4} &= 2 \int \varepsilon \left( -\Delta {\rm Mod}_{{\mathbf W}} - \vec\chi \cdot\nabla {\rm Mod}_{{\mathbf X}}- f'({\mathbf W}) {\rm Mod}_{{\mathbf W}} \right).\end{align*}
\emph{Estimate on $\bf h_1$.}
We claim the following estimate
\begin{equation}\label{bf h_1}
-(1-2\sigma)t {\bf h_1} \le (6-4\delta+C\sigma){\mathcal N_{\Omega}}(t).
\end{equation}
Let
\begin{equation*}
{\rm{I}}=-(1-2\sigma)t\left((-\eta^{2}+|\nabla \varepsilon|^{2})\mathrm{div}\vec \chi -2\sum_{j,k=1}^{5}(\partial_{x_k}\varepsilon)(\partial_{x_j}\varepsilon)\partial_{x_j}\chi_{k}+2\sum_{k=1}^{5}(\partial_{x_k}\varepsilon)\eta\partial_{t}\chi_{k}\right).
\end{equation*}
To obtain \eqref{bf h_1}, we will actually prove the following stronger property
\begin{equation}\label{bf h_1 p.p}
{\rm{I}}\le (6-4\delta)\left(|\nabla \varepsilon|^2 + \eta^2 + 2 (\vec\chi\cdot\nabla \varepsilon )\eta\right) +C\sigma(|\nabla\varepsilon|^{2}+\eta^{2})\quad \forall  x\in \Omega(t).
\end{equation}
Without loss of generality, we consider the following five cases,
\\ \textbf{Case 1:} Let $x\in \Omega_{1,1}\times\Omega_{0,2}\times\Omega_{0,3}\times\Omega_{0,4}\times\Omega_{0,5}$. From \eqref{derchi}, we obtain
\begin{equation}\label{case 1}\left\{\begin{aligned}
& \partial_{x_1} \chi_{1}(t,x)= \frac{1}{(1-2\sigma)t},\quad \partial_{t} \chi_{1}(t,x)= -\frac {x_1}t \frac{1}{(1-2\sigma)t},\\
& \partial_t \chi_{i}(t,x) =0,\quad \nabla \chi_{i}(t,x)=0 \quad \forall i\ne 1.
\end{aligned} \right.\end{equation}
From direct computations and \eqref{case 1}, we obtain
\begin{equation}\label{case 1 p.p}
\begin{aligned}
{\rm{I}}&=\eta^{2}-|\nabla\varepsilon|^{2}+2|\partial_{x_1}\varepsilon|^{2}+2(6-4\delta)(\vec{\chi}\cdot\nabla \varepsilon)\eta-2(5-4\delta)(\partial_{x_1}\varepsilon\cdot\chi_{1})\eta\\
&\ \ -2(6-4\delta)\sum_{i=2}^{5}(\partial_{x_i}\varepsilon\cdot\chi_{i})\eta+2\eta\partial_{x_1}\varepsilon\left(\frac{x_{1}}{t}-\chi_{1}\right).
\end{aligned}
\end{equation}
Finally, using \eqref{spe boud}, \eqref{bd}, \eqref{defchiK}, \eqref{case 1 p.p} and Cauchy-Schwarz inequality, we obtain
\begin{align*}
{\rm{I}} &\le \eta^{2}+|\nabla\varepsilon|^{2}+2(6-4\delta)(\vec{\chi}\cdot\nabla \varepsilon)\eta+C\sigma(|\nabla\varepsilon|^{2}+\eta^{2})\\
& \ \ +(6-4\delta)|\overline{\ell}|\left(\frac{5-4\delta}{|\overline{\ell}|(6-4\delta)}|\nabla\varepsilon|^{2}+\frac{|\overline{\ell}|(6-4\delta)}{(5-4\delta)}\eta^{2}\right)\\
& \le (6-4\delta)\left(|\nabla\varepsilon|^{2}+\eta^{2}+2(\vec{\chi}\cdot\nabla \varepsilon)\eta\right)+C\sigma(|\nabla\varepsilon|^{2}+\eta^{2}).
\end{align*}
 \textbf{Case 2:} Let $x\in \Omega_{1,1}\times\Omega_{1,2}\times\Omega_{0,3}\times\Omega_{0,4}\times\Omega_{0,5}$. From \eqref{derchi}, we obtain
\begin{equation}\label{case 2}\left\{\begin{aligned}
& \partial_{x_i} \chi_{i}(t,x)= \frac{1}{(1-2\sigma)t},\quad \partial_{t} \chi_{i}(t,x)= -\frac {x_i}t \frac{1}{(1-2\sigma)t}\quad \mathrm{for}\ i\in \{1,2\},\\
& \partial_t \chi_{i}(t,x) =0,\quad \nabla \chi_{i}(t,x)=0 \quad \mathrm{for}\ i\in \{3,4,5\}.
\end{aligned} \right.\end{equation}
From direct computations and \eqref{case 2}, we obtain
\begin{equation}\label{case 2 p.p}
\begin{aligned}
{\rm{I}}&=2\eta^{2}-2|\nabla\varepsilon|^{2}+2\sum_{i=1}^{2}|\partial_{x_i}\varepsilon|^{2}+2(6-4\delta)\left(\vec{\chi}\cdot\nabla \varepsilon\right)\eta-2(5-4\delta)\sum_{i=1}^{2}\left(\partial_{x_i}\varepsilon\cdot\chi_{i}\right)\eta\\
&\ \ -2(6-4\delta)\sum_{i=3}^{5}\left(\partial_{x_i}\varepsilon\cdot\chi_{i}\right)\eta+2\sum_{i=1}^{2}\eta\partial_{x_i}\varepsilon\left(\frac{x_{i}}{t}-\chi_{i}\right).
\end{aligned}
\end{equation}
Finally, using~\eqref{spe boud},~\eqref{bd},~\eqref{defchiK},~\eqref{case 2 p.p} and Cauchy-Schwarz inequality, we obtain
\begin{align*}
{\rm{I}}&\le 2\eta^{2}+2(6-4\delta)\left(\vec{\chi}\cdot\nabla \varepsilon\right)\eta+(6-4\delta)|\overline{\ell}|\left(\frac{1}{|\overline{\ell}|}|\nabla\varepsilon|^{2}+|\overline{\ell}|\eta^{2}\right)+C\sigma\left(|\nabla\varepsilon|^{2}+\eta^{2}\right)\\
&\le (6-4\delta)\left(|\nabla\varepsilon|^{2}+\eta^{2}+2(\vec{\chi}\cdot\nabla \varepsilon)\eta\right)+C\sigma\left(|\nabla\varepsilon|^{2}+\eta^{2}\right).
\end{align*}
\textbf{Case 3:} Let $x\in \Omega_{1,1}\times\Omega_{1,2}\times\Omega_{1,3}\times\Omega_{0,4}\times\Omega_{0,5}$. From \eqref{derchi}, we obtain
\begin{equation}\label{case 3}\left\{\begin{aligned}
& \partial_{x_i} \chi_{i}(t,x)= \frac{1}{(1-2\sigma)t},\quad \partial_{t} \chi_{i}(t,x)= -\frac {x_i}t \frac{1}{(1-2\sigma)t}\quad \mathrm{for}\ i\in \{1,2,3\},\\
& \partial_t \chi_{i}(t,x) =0,\quad \nabla \chi_{i}(t,x)=0 \quad \mathrm{for}\ i\in \{4,5\}.
\end{aligned} \right.\end{equation}
From direct computations and \eqref{case 3}, we obtain
\begin{equation}\label{case 3 p.p}
\begin{aligned}
{\rm{I}}&=3\eta^{2}-3|\nabla\varepsilon|^{2}+2\sum_{i=1}^{3}|\partial_{x_i}\varepsilon|^{2}+2(6-4\delta)\left(\vec{\chi}\cdot\nabla \varepsilon\right)\eta\\
&\ \ -2(5-4\delta)\sum_{i=1}^{3}\left(\partial_{x_i}\varepsilon\cdot\chi_{i}\right)\eta-2(6-4\delta)\sum_{i=4}^{5}
\left(\partial_{x_i}\varepsilon\cdot\chi_{i}\right)\eta+2\sum_{i=1}^{3}\eta\partial_{x_i}\varepsilon\left(\frac{x_{i}}{t}-\chi_{i}\right).
\end{aligned}
\end{equation}
Finally, using~\eqref{spe boud},~\eqref{bd},~\eqref{defchiK},~\eqref{case 3 p.p} and Cauchy-Schwarz inequality, we obtain
\begin{align*}
{\rm{I}}&\le 3\eta^{2}-|\nabla \varepsilon|^{2}+2(6-4\delta)(\vec{\chi}\cdot\nabla \varepsilon)\eta+C\sigma(|\nabla\varepsilon|^{2}+\eta^{2})\\
&+(6-4\delta)|\overline{\ell}|\left(\frac{(7-4\delta)}{(6-4\delta)|\overline{\ell}|}|\nabla\varepsilon|^{2}+\frac{(6-4\delta)|\overline{\ell}|}{(7-4\delta)}\eta^{2}\right)\\
&\le (6-4\delta)\left(|\nabla\varepsilon|^{2}+\eta^{2}+2(\vec{\chi}\cdot\nabla \varepsilon)\eta\right)+C\sigma\left(|\nabla\varepsilon|^{2}+\eta^{2}\right).
\end{align*}
\\ \textbf{Case 4:} Let $x\in \Omega_{1,1}\times\Omega_{1,2}\times\Omega_{1,3}\times\Omega_{1,4}\times\Omega_{0,5}$. From \eqref{derchi}, we obtain
\begin{equation}\label{case 4}\left\{\begin{aligned}
& \partial_{x_i} \chi_{i}(t,x)= \frac{1}{(1-2\sigma)t},\quad \partial_{t} \chi_{i}(t,x)= -\frac {x_i}t \frac{1}{(1-2\sigma)t}\quad \mathrm{for}\ i\in \{1,2,3,4\},\\
& \partial_t \chi_{i}(t,x) =0,\quad \nabla \chi_{i}(t,x)=0 \quad \mathrm{for}\ i=5.
\end{aligned} \right.\end{equation}
From direct computations and \eqref{case 4}, we obtain
\begin{equation}\label{case 4 p.p}
\begin{aligned}
{\rm{I}}&=4\eta^{2}-4|\nabla\varepsilon|^{2}+2\sum_{i=1}^{4}|\partial_{x_i}\varepsilon|^{2}+2(6-4\delta)(\vec{\chi}\cdot\nabla \varepsilon)\eta-2(5-4\delta)\sum_{i=1}^{4}(\partial_{x_i}\varepsilon\cdot\chi_{i})\eta\\
&\ \ -2(6-4\delta)(\partial_{x_5}\varepsilon\cdot\chi_{5})\eta+2\sum_{i=1}^{4}\eta\partial_{x_i}\varepsilon\left(\frac{x_{i}}{t}-\chi_{i}\right).
\end{aligned}
\end{equation}
Finally, using~\eqref{spe boud},~\eqref{bd},~\eqref{defchiK},~\eqref{case 4 p.p} and Cauchy-Schwarz inequality, we obtain
\begin{align*}
{\rm{I}}&\le 4\eta^{2}-2|\nabla \varepsilon|^{2}+2(6-4\delta)(\vec{\chi}\cdot\nabla \varepsilon)\eta+C\sigma(|\nabla\varepsilon|^{2}+\eta^{2})\\
& +(6-4\delta)|\overline{\ell}|\left(\frac{(8-4\delta)}{(6-4\delta)|\overline{\ell}|}|\nabla\varepsilon|^{2}
+\frac{|\overline{\ell}|(6-4\delta)}{(8-4\delta)}\eta^{2}\right)\\
&\le (6-4\delta)\left(|\nabla\varepsilon|^{2}+\eta^{2}+2(\vec{\chi}\cdot\nabla \varepsilon)\eta\right)+C\sigma\left(|\nabla\varepsilon|^{2}+\eta^{2}\right).
\end{align*}
\textbf{Case 5:} Let $x\in \Omega_{1,1}\times\Omega_{1,2}\times\Omega_{1,3}\times\Omega_{1,4}\times\Omega_{1,5}$. From \eqref{derchi}, we obtain
\begin{equation}\label{case 5}
 \partial_{x_i} \chi_{i}(t,x)= \frac{1}{(1-2\sigma)t},\quad \partial_{t} \chi_{i}(t,x)= -\frac {x_i}t \frac{1}{(1-2\sigma)t}\quad \mathrm{for}\ i\in \{1,2,3,4,5\}.
\end{equation}
From direct computations and \eqref{case 5}, we obtain
\begin{equation}\label{case 5 p.p}
\begin{aligned}
{\rm{I}}&=5\eta^{2}-5|\nabla\varepsilon|^{2}+2\sum_{i=1}^{5}|\partial_{x_i}\varepsilon|^{2}+2(6-4\delta)(\vec{\chi}\cdot\nabla \varepsilon)\eta\\
&\ \ -2(5-4\delta)\left(\nabla\varepsilon\cdot\vec{\chi}\right)\eta+2\sum_{i=1}^{5}\eta\partial_{x_i}\varepsilon\left(\frac{x_{i}}{t}-\chi_{i}\right).
\end{aligned}
\end{equation}
Finally, using~\eqref{spe boud},~\eqref{bd},~\eqref{defchiK},~\eqref{case 3 p.p} and Cauchy-Schwarz inequality, we obtain
\begin{align*}
{\rm{I}}&\le 5\eta^{2}-3|\nabla \varepsilon|^{2}+2(6-4\delta)(\vec{\chi}\cdot\nabla \varepsilon)\eta+C\sigma(|\nabla\varepsilon|^{2}+\eta^{2})\\
& +(5-4\delta)|\overline{\ell}|\left(\frac{(9-4\delta)}{|\overline{\ell}|(5-4\delta)}|\nabla\varepsilon|^{2}+\frac{|\overline{\ell}|(5-4\delta)}{(9-4\delta)}\eta^{2}\right)\\
&\le(6-4\delta)\left(|\nabla\varepsilon|^{2}+\eta^{2}+2(\vec{\chi}\cdot\nabla \varepsilon)\eta\right)+C\sigma\left(|\nabla\varepsilon|^{2}+\eta^{2}\right).
\end{align*}

\emph{Estimate on $\bf h_2$.} We observe that by the definition of $\vec{\chi}$ in~\eqref{derchi} and the decay of $\nabla W$ and $W$,
\begin{align*}
& \left\| \left(\boldsymbol\ell_k - \vec\chi \right)\cdot \nabla W_k \right\|_{L^{\frac{10}3}}\lesssim t^{-\frac 52}. 
\end{align*}
Thus, by~\eqref{holder1},
$ |{\bf h_2}|\lesssim t^{-\frac 52} \|\varepsilon\|_{L^{\frac{10}3}}^2 
\lesssim C_{0}t^{-\frac {17}2+2\delta}\lesssim C_{0}t^{-7+2\delta}.
$

\emph{Estimate on $\bf h_3$.} Denote
\[
M_k = \left(\frac {\dot \lambda_k }{\lambda_k} - \frac{a_k}{\lambda_k^{\frac 12}t^2}\right)\Lambda W_k + \dot{\mathbf{y}}_k\cdot \nabla W_k \]
so that
${\rm Mod}_{{\mathbf W}} = \sum_k M_k$ and ${\rm Mod}_{{\mathbf X}} = - \sum_k \boldsymbol\ell_k \cdot\nabla M_k$.
Using~\eqref{e:WX} and the definition of $\vec{\chi}$ (see~\eqref{derchi}), we have 
$\|(\boldsymbol\ell_k-\vec\chi)\cdot\nabla M_k\|_{L^2}\lesssim t^{-\frac 92+\delta}$. 
It follows from~\eqref{le:p} that
\[
\left\|{\rm Mod}_{{\mathbf X}} + \vec\chi \cdot \nabla {\rm Mod}_{{\mathbf W}}\right\|_{L^2} \lesssim t^{-\frac 92+\delta},
\]
and thus
\begin{align*}
|{\bf h_3}|=\left| \int \eta\left({\rm Mod}_{{\mathbf X}} + \vec\chi\cdot \nabla {\rm Mod}_{{\mathbf W}}\right) \right|
\lesssim t^{-\frac 92+\delta} \|\eta\|_{L^2} \lesssim C_{0}t^{-7+2\delta}.
\end{align*}

\emph{Estimate on $\bf h_4$.}
By~(i) of Lemma~\ref{surZZ},
$
-\Delta M_k + (\boldsymbol{\ell_{k}}\cdot\nabla)(\boldsymbol{\ell_{k}}\cdot\nabla) M_k - f'(W_k) M_k=0. 
$ Thus,
\begin{align*}
&\left| -\Delta M_k + (\vec\chi\cdot\nabla)\cdot (\boldsymbol{\ell_{k}}\cdot\nabla) M_k - f'({\mathbf W}) M_k\right|\\
&\lesssim \left| ((\vec\chi-\boldsymbol{\ell_{k}})\cdot\nabla)\cdot (\boldsymbol{\ell_{k}}\cdot\nabla) M_k\right| + \left|f'({\mathbf W})-f'(W_k)\right| |M_k|.
\end{align*}
As before, by~\eqref{e:WX},
$
\left\| ((\vec\chi-\boldsymbol{\ell_{k}})\cdot\nabla)\cdot (\boldsymbol{\ell_{k}}\cdot\nabla) M_k\right\|_{L^{\frac{10}7}}\lesssim t^{-\frac 92+\delta}$.
Moreover, by~\eqref{e:n40}
\begin{align*}
&\left| |{\mathbf W}|^{\frac 43} - |W_k|^{\frac 43}\right| |M_k|\\
& \lesssim C_{0}t^{-3+\delta}\left(\sum |W_k|+\sum |v_k|\right)^{\frac 13} \left(\sum_{k'\neq k}|W_{k'}| + \sum |v_k|\right) |W_k|\\
& \lesssim C_{0}t^{-3+\delta}\Bigg[\sum_{k'\neq k''} |W_{k'}|^{\frac 43} |W_{k''}|+t^{-2+\delta} \sum_{k'} |W_{k'}|^{2}\\
&\quad \quad \quad \quad \quad \quad +t^{-\frac{2}{3}+\frac{\delta}{3}} \sum_{k'\neq k''} |W_{k'}|^{1+\frac 29} |W_{k''}| +t^{-\frac{8}{3}+\frac{4}{3}\delta} \sum_{k'} |W_{k'}|^{\frac{17}9} \Bigg],
\end{align*}
and from  Claim~\ref{WW},
\[
\left\| \left(f'({\mathbf W})-f'(W_k)\right) M_k\right\|_{L^{\frac {10}7}}\lesssim t^{-5+2\delta}.
\]
Therefore, using~\eqref{le:p},
\[
\left\| -\Delta {\rm Mod}_{{\mathbf W}} - \vec{\chi}\cdot\nabla {\rm Mod}_{{\mathbf X}}- f'({\mathbf W}) {\rm Mod}_{{\mathbf W}} \right\|_{L^{\frac{10}7}}
\lesssim t^{-\frac 92+\delta}.
\]
It follows that by~\eqref{eq:BS},
\[
|{\bf h_4}|\lesssim \|\varepsilon\|_{L^{\frac{10}3}} \left\| -\Delta {\rm Mod}_{{\mathbf W}} -\vec{\chi}\cdot \nabla {\rm Mod}_{{\mathbf X}}- f'({\mathbf W}) {\rm Mod}_{{\mathbf W}} \right\|_{L^{\frac{10}7}}
\lesssim C_{0} t^{-7+2\delta}.
\]

In conclusion, using~\eqref{coer}, for $\sigma$ small, and $T_0$ large,
\[
-\frac d{dt} \mathcal H
\leq \frac {(6-4\delta+C \sigma)} t {\mathcal N_{\Omega}} + O(C_{0}t^{-7+2\delta})
\leq \frac{6-3\delta}{t} \mathcal H+ O(C_{0}t^{-7+2\delta})
\]
and the proof of Lemma~\ref{mainprop} is complete.
\end{proof}
\subsection{Parameters and energy estimates}\label{s4}
The following result, mainly based on Lemma~\ref{mainprop}, improves all the estimates in~\eqref{eq:BS}, except
the ones on $(z_k^-)_k$.
\begin{lemma}[Closing estimates except $(z_k^-)_k$]\label{le:bs1}
	For $C_0>0$ large enough, for all $t\in[T^*,S_n]$,
	\begin{equation}\label{eq:BSd}\left\{\begin{aligned}
	& |\lambda_k(t)-\lambda_k^\infty|\leq \frac {C_0} 2t^{-1},\quad |\mathbf{y}_k(t)-\mathbf{y}_k^\infty|\leq \frac {C_0}2t^{-1},\\
	&|z_k^+(t)|^2\leq \frac 12 t^{-7},\quad 
	\left\|\vec \varepsilon(t)\right\|_{\dot H^1\times L^2}\leq \frac{C_{0}}{2} t^{-3+\delta}.\end{aligned}\right. \end{equation}
\end{lemma}
\begin{proof}
	\textbf{Step 1.} Estimate on $(\lambda_{k}(t),\mathbf{y}_{k}(t))_{k\in\{1,\dots,K\}}$.
	From~\eqref{le:p} and~\eqref{eq:BS}, we have
	\begin{equation}\label{dly}
\left|\frac {\dot \lambda_k}{\lambda_k}\right|+|\dot {\mathbf y}_k|\leq C t^{-2}
\end{equation}
	where the constant $C$ depends on the parameters of the $K$ solitons, but not on $C_0$. Using ~\eqref{modu3} and ~\eqref{dly} and taking $C_{0}$ large enough, we obtain
	\begin{align*}
|\lambda_{k}(t)-\lambda_{k}^{\infty}|&\le|\lambda_{k}(t)-\lambda(S_{n})|+|\lambda(S_{n})-\lambda_{k}^{\infty}|\\
& \le |\int_{t}^{S_{n}}\dot{\lambda_{k}}(s)ds|+\frac{1}{S^{\frac{7}{2}}_{n}}\le \int_{t}^{S_{n}}\frac{C}{s^{2}}ds+\frac{1}{S^{\frac{7}{2}}_{n}}\le \frac{C_{0}}{2t}.
\end{align*}
Using again ~\eqref{modu3} and ~\eqref{dly} and taking $C_{0}$ large enough, we obtain
\begin{align*}
|\mathbf{y}_{k}(t)-\mathbf{y}_{k}^{\infty}|&\le|\mathbf{y}_{k}(t)-\mathbf{y}(S_{n})|+|\mathbf{y}(S_{n})-\mathbf{y}_{k}^{\infty}|\\
& \le |\int_{t}^{S_{n}}\dot{\mathbf{y}_{k}}(s)ds|+\frac{1}{S^{\frac{7}{2}}_{n}}\le \int_{t}^{S_{n}}\frac{C}{s^{2}}ds+\frac{1}{S^{\frac{7}{2}}_{n}}\le \frac{C_{0}}{2t}.
\end{align*}

\textbf{Step 2.} Bound on the energy norm. From ~\eqref{modu3} and ~\eqref{boun}, we have
\begin{equation}\label{bHS}
\mathcal H(S_n) \lesssim {S_n^{-7}}.
\end{equation}
Integrating~\eqref{time} on $[t,S_n]$, and using \eqref{bHS}, we obtain,
\begin{align*}
t^{6-3\delta}\mathcal{H}(t)&=-\int_{t}^{S_{n}}\frac{d}{ds}(s^{6-3\delta}\mathcal{H}(s))ds+S_{n}^{6-3\delta}\mathcal{H}(S_{n})\lesssim \int_{t}^{S_{n}}\frac{C_{0}}{s^{1+\delta}}ds\lesssim\frac{C_{0}}{t^{\delta}}+\frac{1}{S_{n}^{1+3\delta}}.
\end{align*}
It follows that, $\mathcal{H}(t)\lesssim C_{0}t^{-6+2\delta}$. Using~\eqref{coer}, we conclude that
$
\|\vec \varepsilon\|_{\dot H^1\times L^2} \lesssim \sqrt{C_{0}}t^{-3+\delta}$. The estimate follows for $C_{0}$ large enough.

\textbf{Step 3.} Estimate on $\big(z^{+}_{k}(t)\big)_{k\in\{1,\dots,K\}}$. Let $\beta_k=\frac{\sqrt{\lambda_0}}{\lambda_k^{\infty}}(1-|{\boldsymbol{\ell}}_k|^2)^{1/2}>0$.
Then, from~\eqref{le:z} and~\eqref{eq:BS}, 
\begin{equation}\label{bdz}
\frac d{dt} \left( e^{-\beta_k t} z_k^+\right)\lesssim C_{0} e^{-\beta_k t} t^{-4+\delta}. 
\end{equation}
Integrating ~\eqref{bdz} on $[t,S_n]$, and using~\eqref{modu:2}, we obtain
$- z_k^+(t) \lesssim {t^{-4+\delta}}$. Using the same estimate for $-e^{-\beta_{k}t} z_{k}^{+}$, we obtain the conclusion for $T_{0}$ large enough.
	\end{proof}
As in~\cite{CMM,CMkg,MMwave1, MMwave2}, the parameters $z_k^-$ require a specific argument.
\begin{lemma}[Control of unstable directions in~\cite{MMwave2}]\label{le:bs2}
	There exist 
	$(\xi_{k,n})_{k}\in \mathcal B_{{\mathbb{R}}^K}(S_n^{-7/2})$ such that, for $C^*>0$ large enough, $T^*((\xi_{k,n})_k)=T_0$.
	In particular, let $(\zeta_n^\pm)$ be given by Claim~\ref{le:modu2} from such $(\xi_{k,n})_{k}$,
	then the solution $u_n$ of~\eqref{defun} satisfies~\eqref{eq:un}.
\end{lemma}
\begin{proof}
See proof of Lemma 5.8 in~\cite{MMwave2}.
\end{proof}
	
Estimates~\eqref{eq:un} follow directly from the estimates~\eqref{eq:BS} on $\varepsilon(t)$, $\lambda_k(t)$, $\mathbf{y}_k(t)$.

\subsection{End of the proof of Proposition~\ref{pr:s5}}\label{s5}
We refer to~\cite{MMwave2} for the proofs of the $\dot{H}^{2}\times\dot{H}^{1}$ bound and~\eqref{cpq} (\S\ref{s4} and \S\ref{s5} of~\cite{MMwave2}).

\end{document}